\documentclass[12pt]{amsart}
\usepackage[paperwidth=210mm,paperheight=297mm,inner=3cm,outer=3cm,top=2.5cm,bottom=2.6cm,marginparwidth=1.8cm]{geometry}

\usepackage{amssymb}
\usepackage{amsmath}
\usepackage{amsthm}
\usepackage{mathtools}
\usepackage{bm}
\usepackage{mathrsfs}

\usepackage[T1]{fontenc}
\usepackage{xcolor}

\usepackage{hyperref}

\usepackage[utf8]{inputenc}
\usepackage{float}

\usepackage{graphicx}
\usepackage{subfig}

\usepackage{cleveref}

\theoremstyle{plain}
\newtheorem{thm}{Theorem}
\numberwithin{thm}{section}
\newtheorem{prob}{Problem}

\newtheorem{cor}[thm]{Corollary}
\newtheorem{con}[thm]{Conjecture}
\theoremstyle{definition}
\newtheorem{ex}[thm]{Example}
\newtheorem{rem}[thm]{Remark}

\def\C{\mathbb{C}}
\def\R{\mathbb{R}}
\def\Q{\mathbb{Q}}

\newcommand{\cH}{\mathscr{H}}

\begin{document}

\bibliographystyle{abbrv}

\title[Computing totally real hyperplane sections]{Computing totally real hyperplane sections and linear
series on algebraic curves}

\author{Huu Phuoc Le}
\email{huu-phuoc.le@lip6.fr}
\address{Sorbonne Université, LIP6, CNRS, Équipe PolSys, F-75252, Paris Cedex 05, France}

\author{Dimitri Manevich}
\email{dimitri.manevich@math.tu-dortmund.de}
\address{Fakultät für Mathematik, Technische Universität Dortmund, D-44227 Dortmund, Germany}

\author{Daniel Plaumann}
\email{daniel.plaumann@math.tu-dortmund.de}
\address{Fakultät für Mathematik, Technische Universität Dortmund, D-44227 Dortmund, Germany}

\begin{abstract}
Given a real algebraic curve, embedded in projective space, we study
the computational problem of deciding whether there exists a
hyperplane meeting the curve in real points only. More generally,
given any divisor on such a curve, we may ask whether the
corresponding linear series contains an effective divisor with totally
real support. This translates into a particular type of parametrized
real root counting problem that we wish to solve exactly. On the other
hand, it is known that for a given genus and number of real connected
components, any linear series of sufficiently large degree contains a
totally real effective divisor. Using the algorithms described in this
paper, we solve a number of examples, which we can compare to the best
known bounds for the required degree.
\end{abstract}



\maketitle

\section*{Introduction}
Given a real algebraic curve $X$ of degree $d$ embedded into some
projective space, we consider the computational problem of deciding
whether there exists a real hyperplane meeting $X$ in a prescribed
number $r$ of real points, counted with multiplicity. Of particular
interest is the case $r=d$, i.e., hyperplanes meeting $X$ in real
points only. More generally, given any divisor $D$ on $X$ defined over
$\mathbb{R}$, and thus consisting of real points and complex-conjugate
pairs, we may ask whether the linear series $|D|$ contains an
effective divisor with totally real support. (The first question is
the special case when $D$ is a hyperplane section of a suitably
embedded curve.)
\begin{figure}[ht]
  \centering \includegraphics[width=6cm]{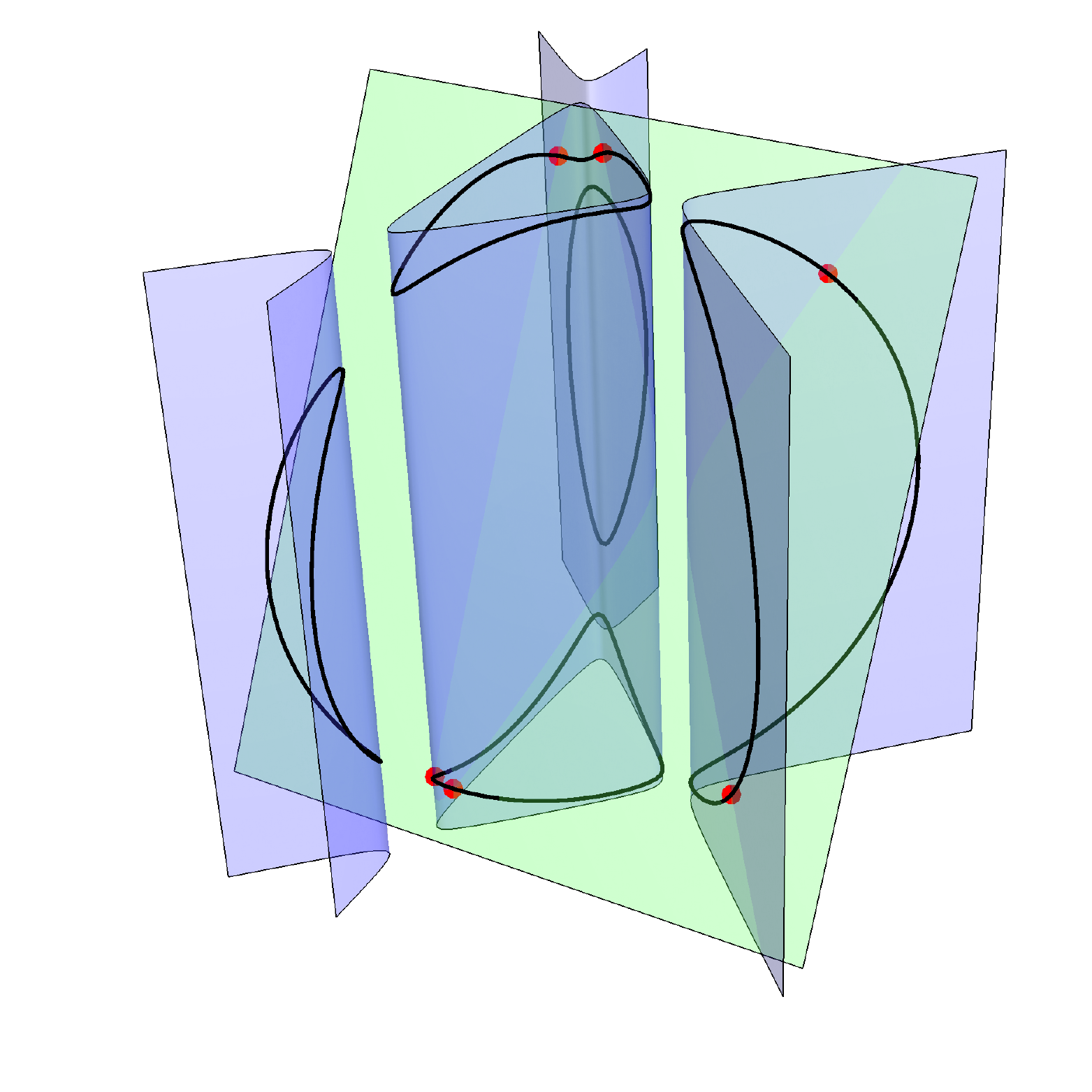}\\
  \caption{A real space curve of degree $6$ with a totally real hyperplane section.}
\end{figure}
A number of general results have been obtained in this direction: 
The answer is known to be positive for any divisor of
sufficiently high degree (see \cite{krasnov} and
\cite{scheiderer}). However, the precise degree required, relative to
the genus of $X$, is the subject of several results and conjectures,
some of which we will investigate from a computational point of
view. Explicit bounds are only known if the real locus $X(\mathbb{R})$
has many connected components (so-called $M$-curves or
$(M-1)$-curves), by results due to Huisman \cite{huisman1} and Monnier
\cite{monnier}. On the other hand, very little is known about curves
whose number of connected components is not close to maximal. Of
course, the computational problem makes sense for any given curve and
divisor, regardless of whether or not there is a general result
covering all curves and divisors of the given kind.

It comes down to ``solving'' polynomial systems whose coefficients
depend on parameters. More precisely, we consider the coefficients of
the equation defining the hyperplane as \emph{parameters}. One then
associates a hyperplane to a point in the space of parameters. The
number of real points at the intersection of the considered hyperplane
with the curve may vary depending on the parameters, while the number
of complex intersection points between the curve and the hyperplane is
equal to the degree $d$ for \emph{generic} values of the
parameters. (If the points are counted with intersection
multiplicities and the curve is not contained in a hyperplane, this
complex intersection number is equal to $d$ for \textit{all} values of
the parameters.)  Hence, from a computational point of view, we are
considering a polynomial system, depending on parameters such that,
when these parameters take generic values, the solution set over the
complex numbers is finite. When the input system generates a radical
ideal, the algorithm we use, which is detailed in \cite{LeSa20},
computes a partition of a dense semi-algebraic subset of the space of
parameters into open semi-algebraic sets such that the number of real
\emph{simple} solutions (i.e., without multiplicities) to the input
system is invariant for any point chosen in one of these sets. To do
this, we compute a symmetric matrix called the \emph{parametric
Hermite matrix}, whose entries are polynomials depending on the
parameters and such that, after specialization, its signature
coincides with the number of real solutions to the specialized
system. This allows us to classify the possible number of real roots
to the input system with respect to the parameters. 

Our main findings can be summarized as follows.

  \begin{itemize}
    \item[1.] There exist canonical curves $X$ in $\mathbb{P}^3$ with one or two ovals which do not allow simple totally real hyperplane sections (Example \ref{ex1p3}).
    \item[2.] There exists a curve $X$ in $\mathbb{P}^3$ of genus two and degree five having one oval which does not allow a simple totally real hyperplane section (Example \ref{exvar}).
    \item[3.] There are infinitely many plane quartics $X$ with many ovals possessing a (complete) linear series of degree four which does not contain a totally real divisor (Example \ref{quarticex}).
    \item[4.] For every $d \geq 3$ and every number $1 \leq s \leq g+1$ with
    $g=\frac{(d-1)(d-2)}{2}$, there exists a plane curve $X$ of degree
    $d$, genus $g$ and having $s$ branches such that the linear series of
    lines $\vert L \vert$ is totally real (Theorem \ref{Thm:RealHarnack}).
  \end{itemize}

The paper is structured as follows. Section~\ref{sec:prelim} is
devoted to preliminaries; we recall basic definitions and properties.
Section~\ref{Sec:Algorithms} describes the algorithm we use to solve
parametric polynomial systems representing the hyperplane
sections. In Section~\ref{sec3}, we apply our computational methods to (canonical) space curves. In Section \ref{Sec:planequartics}, we determine the real divisor bound for certain plane quartics.

\bigskip
 \noindent \textbf{Acknowledgements.}
We would like to thank Matilde Manzaroli for helpful discussions
concerning the proof of Theorem \ref{Thm:RealHarnack}. We would like
to thank Mohab Safey El Din for important discussions and remarks on
the computations. Daniel Plaumann was partially supported through DFG
grant no.~426054364.

\section{Preliminaries}\label{sec:prelim}
By a \emph{real (algebraic) curve} $X$, we mean an integral, smooth
and projective variety of dimension $1$ defined over $\mathbb{R}$ such
that the set $X\left(\mathbb{R}\right)$ of real points is non-empty
(and therefore Zariski-dense in $X$), unless any of these assumptions
is explicitly dropped. Note that a smooth curve is a curve without any
singularities, real or complex. In particular, the set $X(\mathbb{R})$
is an analytic manifold and decomposes into a finite number of
connected components, which are called the \emph{(real) branches} of
$X$. Each branch is diffeomorphic to a circle $S^1$. By Harnack's
Inequality \cite{harnack}, we have $s \leq g+1$, where $s$ is the
number of branches and $g$ is the genus of $X$.

If $X$ is embedded into the projective space $\mathbb{P}^n$, a branch
of $X$ is an \emph{oval} if its homology class in
$H_{1}(\mathbb{P}^{n}(\mathbb{R}), \mathbb{Z}/2)$ is trivial, and a
\emph{pseudo-line} otherwise. Equivalently, ovals are those branches
of $X$ that meet every real hyperplane in $\mathbb{P}^n$ in an even
number of real points (counted with multiplicities), while pseudo-lines meet hyperplanes in an odd
number of points. In particular, a pseudo-line has non-empty
intersection with any hyperplane.

We fix some notation and terminology concerning divisors on curves. As a general reference (covering also curves defined over non-algebraically closed fields), we suggest \cite[Ch.~7]{liu}. A \emph{divisor} on $X$ is a formal $\mathbb{Z}$-linear combination of points
$$D= \sum_{i=1}^{m}n_iP_i \qquad \left(m \in \mathbb{N}_0, n_i \in \mathbb{Z}, P_i\in X\right). $$
Assuming that the points $P_1,\dots,P_m$ are distinct and $n_i\neq 0$ for all $i$, the set $\{P_1,\dots,P_m\}$  is called the \emph{support} of the divisor, the numbers $n_1,\dots,n_m$ the \emph{multiplicities} and $\sum_{i=1}^m n_i$ the \emph{degree}. If all multiplicities in $D$ are nonnegative, the divisor $D$ is called \emph{effective}. If all multiplicites are equal to $1$, the divisor is called \emph{simple}. The support of a divisor on a real curve may consist of real or complex points. However, we will only consider divisors that are \emph{defined over $\R$} and hence conjugation-invariant, i.e., for any point in the support, its complex-conjugate appears with equal multiplicity. In particular, the non-real part of a divisor is of even degree.

For any non-zero real rational function $f\in\R(X)$ on $X$, the divisor of zeros and poles (counted with positive or negative mutiplicities, respectively) is denoted ${\rm div}(f)$. Two divisors $D$ and $E$ are called \emph{linearly equivalent} if $E=D+{\rm div}(f)$ for some $f\in\R(X)^\ast$. The principal divisors ${\rm div}(f)$ have degree $0$, hence linear equivalence preserves the degree. The \emph{complete linear series associated to $D$} is the set of effective divisors on $X$ which are linearly equivalent to $D$ and is denoted $|D|$. A complete linear series carries the structure of a projective space. Any projective subspace of a complete linear series is called a \emph{linear series}. If a point is contained in the support of all divisors in a given linear series, it is a called a \emph{base point}, and the union of all such points is the \emph{base locus}. A linear series is called \emph{base-point-free} if its base locus is empty.

For a real curve $X$ embedded into projective space $\mathbb{P}^n$ with degree $d$, any hypersurface $Z\subset\mathbb{P}^n$ of degree $e$ not containing $X$ defines an effective intersection divisor $X\cdot Z$ of degree $de$. The set of all intersections with hypersurfaces of a fixed degree forms a linear series on $X$, which may or may not be complete. Clearly, such a linear series is always base-point-free. 

\medskip
An effective divisor $D$ is called \emph{totally real} if its support consists of real points only. For the sake of brevity, we call any linear series \emph{totally real} if it contains a totally real (effective) divisor.
After discussing the algorithms in Section \ref{Sec:Algorithms}, we
will examine the following problems.

\begin{prob}\label{probdivisor}
Given a real curve $X$, determine the smallest natural number $N(X)
\in \mathbb{N}^*$ such that any divisor of degree at least $N(X)$ is
linearly equivalent to a totally real divisor.
\end{prob}

We call $N(X)$ the \emph{real divisor bound} of $X$. It was shown by
Krasnov \cite[Thm.~2.2]{krasnov} and Scheiderer
\cite[Cor.~2.10]{scheiderer} that the real divisor bound
is always finite. Furthermore, upper and lower bounds for $N(X)$ were found by
Huisman \cite{huisman1} and Monnier \cite{monnier}, which depend on the
genus $g$ of $X$ only. For example, if $X$ is an $M$-curve or an
$(M-1)$-curve, then we have $N(X) \leq 2g-1$. However, it seems difficult to find upper bounds for curves with few branches. 

An easy
way to determine lower bounds for $N(X)$ is to find a linear series
with a pair of complex-conjugate base points, i.e., a non-real point that is fixed throughout the linear series. With that idea, Monnier \cite[Cor.~6.2]{monnier} proved the inequality $N(X) \geq g+1$ for a curve $X$ with
any number of branches. It seems that no such lower bound is known when considering only base-point-free linear series.
At the end of Section
\ref{Sec:planequartics}, we will construct an example of such a linear series on a
plane quartic curve.

\begin{prob}\label{probhyperplane}
Given a real curve $X$ embedded in projective space, decide whether
the linear series of hyperplanes contains a totally real hyperplane
section.
\end{prob}

Note that according to Bertini's Theorem, the generic element of
a linear series on $X$ is simple away from the base locus
(see \cite[Ch.~1, p.~137]{principles}). However, it may happen that a linear
series contains a totally real divisor, but no simple
such divisor. For example, the linear series of lines on the
plane quartic $X= \mathcal{V}_+(x^4+y^4-z^4) \subset \mathbb{P}^{2}
$ contains the totally real line section
$$X \cdot \mathcal{V}_{+}(x-z) = 4\cdot [1:0:1],$$
but it is easy to see that there is no simple totally real line section.

\begin{prob}\label{probsmoothdivisor}
Given a real curve $X$, determine the smallest natural number $N'(X)
\in \mathbb{N}^*$ such that any divisor of degree at least $N'(X)$ is
linearly equivalent to a simple totally real divisor.
\end{prob}

We call $N'(X)$ the \emph{simple real divisor bound} of $X$. It was
first introduced in \cite[p.~29]{bardet}. Obviously, we have $N(X)
\leq N'(X)$ and a first non-trivial result comparing $N(X)$ and
$N'(X)$ is obtained in \cite[Prop.~2.1.2]{bardet}, namely $N'(X) \leq
2N(X)$. However, it appears to be unknown if $N(X)$ and $N'(X)$ can ever actually be different.

One reason for the importance of the simple real divisor bound comes from
the possibility of transfering results from smooth to singular curves (see
\cite[Thm.~4.3]{singularmonnier}). Basically, our algorithm computes
simple totally real hyperplane sections. When we are mainly interested
in the non-existence of totally real divisors within a linear series,
i.e., in lower bounds for $N(X)$, we modify the algorithm in a way
explained in \Cref{Sec:Algorithms} to handle totally real hyperplane
sections in general.

\section{Algorithm for solving parametric systems}
\label{Sec:Algorithms}

We consider as input $\bm{f} = (f_1, \ldots, f_s)$ in
$\Q[\bm{y}][\bm{x}]$ with $\bm{y} = (y_1, \ldots, y_t)$ and $\bm{x} =
(x_1, \ldots, x_n)$. We assume that there exists a non-empty
Zariski-open subset $\Omega$ of $\C^t$ such that the number of complex
solutions to $\bm{f}(\eta, \bm{x})$ is finite for every $\eta \in
\Omega$ and that $\bm{f}$ generates a radical ideal in
$\Q(\bm{y})[\bm{x}]$. Below, we describe the main ingredients which
allow us to classify the real roots of the system $\bm{f}$, i.e., to
compute semi-algebraic formulas defining a partition
$S_1\cup\cdots\cup S_r$ of a dense semi-algebraic subset of $\R^t$
such that for a given $1\leq i \leq r$ and all $\eta\in S_i$, the
number of real roots to $\bm{f}(\eta, \bm{x})$ is invariant.

To do that, we rely on well-known properties of {\em Hermite's quadratic
form} to count the real roots of zero-dimensional ideals; see
\cite{Hermite56}.  Basically, given a zero-dimensional ideal
$I\subset \Q[\bm{x}]$, Hermite's quadratic form is defined on the finite
dimensional $\Q$-vector space $A \coloneqq \Q[\bm{x}]/I$ by
\begin{align*}
  A \times A & \to \Q \\ 
  (h,k)  & \mapsto {\rm trace}(\mathscr{L}_{h\cdot k}),
\end{align*}
where $\mathscr{L}_{h\cdot k}$ denotes the endomorphism $p\mapsto
h\cdot k\cdot p$ of $A$.

The number of distinct real (resp.~complex) roots of the algebraic set
defined by $I$ equals the signature (resp.~rank) of Hermite's
quadratic form; see e.g.~\cite[Thm.~4.101]{BPR}. Given a basis of
$\Q[\bm{x}]/I$, such a quadratic form is represented by a symmetric
matrix of size $\delta \times \delta$, where $\delta$ is the degree of
$I$. Hence, the signature of Hermite's quadratic form can be computed
once a matrix representation, which we call Hermite's matrix, of this
quadratic form is known \cite[Algo.~8.18]{BPR}.

In~\cite{LeSa20}, we slightly extend the definition of Hermite's
quadratic form and Hermite's matrix to the context of parametric
systems; we call them parametric Hermite quadratic form and parametric
Hermite matrix. This is easily done since the ideal of
$\Q(\bm{y})[\bm{x}]$ generated by $\bm{f}$, considering $\Q(\bm{y})$
as the base field, has dimension zero. 

We also establish a specialization property for this parametric
Hermite matrix: we identify a polynomial $\bm{w}_{\infty} \in
\Q(\bm{y})$ such that, when specializing the parameters $\bm{y}$ in
the Hermite matrix to a point $\eta\in \R^t$ where
$\bm{w}_{\infty}(\eta) \ne 0$, we obtain a Hermite matrix representing
Hermite's quadratic form in $\Q[\bm{x}]/ \langle \bm{f}(\eta, \bm{x})
\rangle$.

Hence, such a parametric Hermite matrix allows us to count
respectively the number of distinct real and complex roots at any
parameters outside a strict algebraic sets of $\R^t$ by evaluating the
signature and rank of its specialization.

Based on the aforementioned specialization property, we design an
algorithm for solving parametric systems as follows.

\begin{itemize}
\item[(a)] We start by computing a parametric Hermite matrix
$\mathscr{H}$ associated to $\bm{f} \subset \Q[\bm{y}][\bm{x}]$. Note
that this requires computations over the quotient algebra
$\Q(\bm{y})[\bm{x}] / \langle \bm{f} \rangle$ through the theory of
Gr\"obner bases.

From the matrix $\mathscr{H}$, we derive two polynomials:
$\bm{w}_{\infty}$ encoding the non-specialization locus of
$\mathscr{H}$ and $\bm{w}_{\mathscr{H}}$ which is basically the
numerator of $\det(\mathscr{H})$. The product $\bm{w}_{\infty} \cdot
\bm{w}_{\cH}$
is denoted by $\bm{w}$.

\item[(b)] Next, we compute a set of sample points $\{\bm{a}_1,
\ldots, \bm{a}_{\ell}\}$ in the connected components of the
semi-algebraic set of $\R^t$ defined by $\bm{w}_{\infty} \neq 0$ and
$\bm{w}_{\cH} \neq 0$ where $\bm{w}_{\cH}$ is derived from $\mathscr{H}$.

This is done through the so-called critical point method (see e.g.
\cite[Ch.~12]{BPR} and references therein) which are adapted to
obtain practically fast algorithms following \cite{SaSc03}.

By \cite[Prop.~21]{LeSa20}, for any $\eta$ varying over the connected
component containing a sample point $\bm{a}_i$, the number of real
solutions to $\bm{f}(\eta, \bm{x})$ is the same as the number of real
solutions to $\bm{f}(\bm{a}_i, \bm{x})$.

\item[(c)] For $1\leq i \leq \ell$, evaluate the signature of the
specialized Hermite matrix $\mathscr{H}(\bm{a}_i)$, which gives the
number $r_i$ of real solutions to $\bm{f}(\bm{a}_i, \bm{x})$.
\end{itemize}

In most of the cases, the algorithm above is sufficient to compute a
hyperplane that intersects the given curve at only real points if such
a hyperplane exists. From a computational point of view, Step
\emph{(b)} is usually the most expensive: the polynomial $\bm{w}$ it
takes as input may have large degree since it may be exponential in
the number of variables $n$ (but polynomial in the maximum degree of
the input polynomials).

Note also that the resulting classification holds only for the subset
of the space of parameters where $\bm{w} \ne 0$. The vanishing locus
of $\bm{w}$ contains points above which either the matrix $\mathscr{H}$
does not specialize well ($\bm{w}_{\infty} = 0$) or $\bm{f}$ has
multiple roots ($\bm{w}_{\cH} = 0$).

Theoretically, a complete root classification, i.e., the number
of real solutions of $\bm{f}$ for every $\eta \in \R^t$ can be
obtained using a similar routine. This consists of classifying the
solutions of $\bm{f}$ over the vanishing locus of $\bm{w}$. There are
several possible approaches, for instances, computing over
the algebraic extension $\mathbb{Q}[\bm{y}]/\langle \bm{w} \rangle$ or
calling the algorithm above on with $\bm{w}$ added to the input
system. The first approach usually leads to high arithmetic costs
while the second induces Hermite matrices of large size (depending on
the degree of $\bm{w}$). One can also try to compute the sign
conditions of the leading principal minors of $\cH$ while imposing a
rank deficiency on the matrix. This results in deciding the emptiness
of a semi-algebraic set whose defining atoms are minors of the Hermite
matrix. To the best of our knowledge, these methods can be
computationally difficult in practice.

However, in the examples we consider in this paper, the polynomials
$\bm{w}_{\infty}$ correspond to the hyperplanes which intersect the
given curves at infinity and are factorized into polynomials of small
degree (at most $3$). Thus, they can be treated by calling the
algorithm on the input $\bm{f}$ adding each factor of
$\bm{w}_{\infty}$. Looking closer, these factors can be simplified
before being sent to the above algorithm to accelerate the
computation. For examples, linear factors can be handled through
substitutions of variables or the quadratic factors which are sums of
squares can be replaced by linear equations. Further, these processes
will be explained in detail for each example.

On the contrary, handling the solutions of $\bm{w}_{\cH}$, where the
system $\bm{f}$ has multiple roots, requires an expensive
computation. Therefore, our algorithm is limited at the moment to
computing simple totally real hyperplane sections, i.e., the
intersection has only simple points.

In the particular case of one-parameter (see the examples in
Section~\ref{Sec:planequartics}), we can obtain easily the complete
root classification by evaluating the signs of leading principal
minors of the matrix $\mathscr{H}$ at real solutions of $w$ using
exact algorithms for real root isolation \cite{XiaYang02,KoRoSa16}.

We illustrate the algorithm above by the following example.

\begin{ex}
We consider the parametric system 
\[\bm{f} = \{x_1^2+x_2^2-y, x_1^2+x_1x_2-yx_2+x_1+y^2\},\]
where $(x_1,x_2)$ are variables and $y$ is the parameter. Following
\cite[Algo. 2]{LeSa20}, we obtain the basis $\{1,x_2,x_1,x_2^2\}$ for
the quotient ring $\mathbb{Q}[y][x_1,x_2] / \langle \bm{f} \rangle$
and the symmetric Hermite matrix associated to this basis
\[ \cH = \begin{bmatrix}
4 & -y-1 & y-1 & 2y^2+5y \\
* & 2y^2+5y & -3y^2-y+1 & y^3/2-6y^2-3y+1/2 \\
* & *  & -2y^2-y & 7y^3/2+4y^2-y-1/2 \\
* & *  & * & -5y^4/2+5y^3+23y^2/2+y-1/2 \\
\end{bmatrix}.\]
The non-specialization polynomial $\bm{w}_{\infty}$ in this example is
identically $1$. The determinant of this Hermite matrix is 
\[\bm{w} = \bm{w}_{\cH} =  41y^8+43y^7-59y^6-204y^5-60y^4+20y^3+4y^2-y.\]
This polynomial has two real solutions: $0$ and $\tilde{y} \approx
1.714$. So, the semi-algebraic set defined by $\bm{w} \ne 0$ has three
connected components and the number
of distinct real solutions of $\bm{f}$ is invariant over each of those
connected components. More precisely,
\begin{align*}
y < 0 : \qquad & \bm{f} \text{ has }0\text{ real solution,}\\
0 < y < \tilde{y} : \qquad & \bm{f} \text{ has }2\text{ real
solutions,}\\
\tilde{y} < y : \qquad & \bm{f} \text{ has }0\text{ real solution.}
\end{align*}
Now we study the roots of $\bm{f}$ over two real roots of $w$.

We specialize $y$ to $0$ in the leading principal minors of
$\mathscr{H}$ and obtain the sign sequence $(1,-1,-1,0)$. Thus, the
system has three distinct complex solutions but only one real solution
when $y = 0$.

For $y = \tilde{y}$, we obtain the sign sequence $(1,-1,1,0)$ for the
leading principal minors specialized at $y = \tilde{y}$. Therefore,
the system has three distinct complex solutions but no real solution.\hfill$\triangle$
\end{ex}

Further, we will use this algorithm for solving parametric polynomial
systems arising in the computation of totally real hyperplane
sections.

\section{Totally real hyperplane sections}\label{sec3}
The possibilities of our computational approach can by shown by the
following examples. We point out that $X$ is always assumed to be a
real curve and $g$ stands for the genus of $X$. If $X$ is a real
rational or real elliptic curve, it is not hard to see that $N(X)=1$. Hence, we assume $g\geq 2$.

We first consider canonical curves: If $X \subset \mathbb{P}^{g-1}$ is
a canonical curve having $s \geq g-1$ branches, then the canonical
linear series, which is equal to the hyperplane linear series, is
totally real. Since there are no canonical curves of genus $g \le 2$,
the minimal examples are plane quartic curves. In this case, the
question of whether a plane quartic curve consisting of only one oval 
possesses a totally real line section is related to the undulation
invariant (see \cite[Thm.~4.2]{undulation}). We therefore look at
canonical curves in $\mathbb{P}^3$.

\begin{ex}\label{ex1p3}
In this example, we consider a finite sequence of canonical curves
$X_k$ in $\mathbb{P}^3$; these curves arise as complete intersections
of a cubic and a quadric. Their genus is $4$ and their degree is
$6$. In non-homogeneous coordinates, we fix the real cubic polynomial
$f= (x+3)(x-y-3)(x+y-3)-2$.\\

$1.$ We set $g_5=x^2+y^2+z^2-100$, $g_4=(x+3)^2+(y+2)^2+z^2-60$ and
$g_3=x^2+y^2+z^2-50$. Let $X_k$ be the projective curve defined by the
affine ideal $I_k=\langle f,g_k \rangle$ for $k=3,4,5$. The curve
$X_k$ has $k$ ovals.
\begin{figure}[ht]
    \centering
    \includegraphics[width=4.3cm]{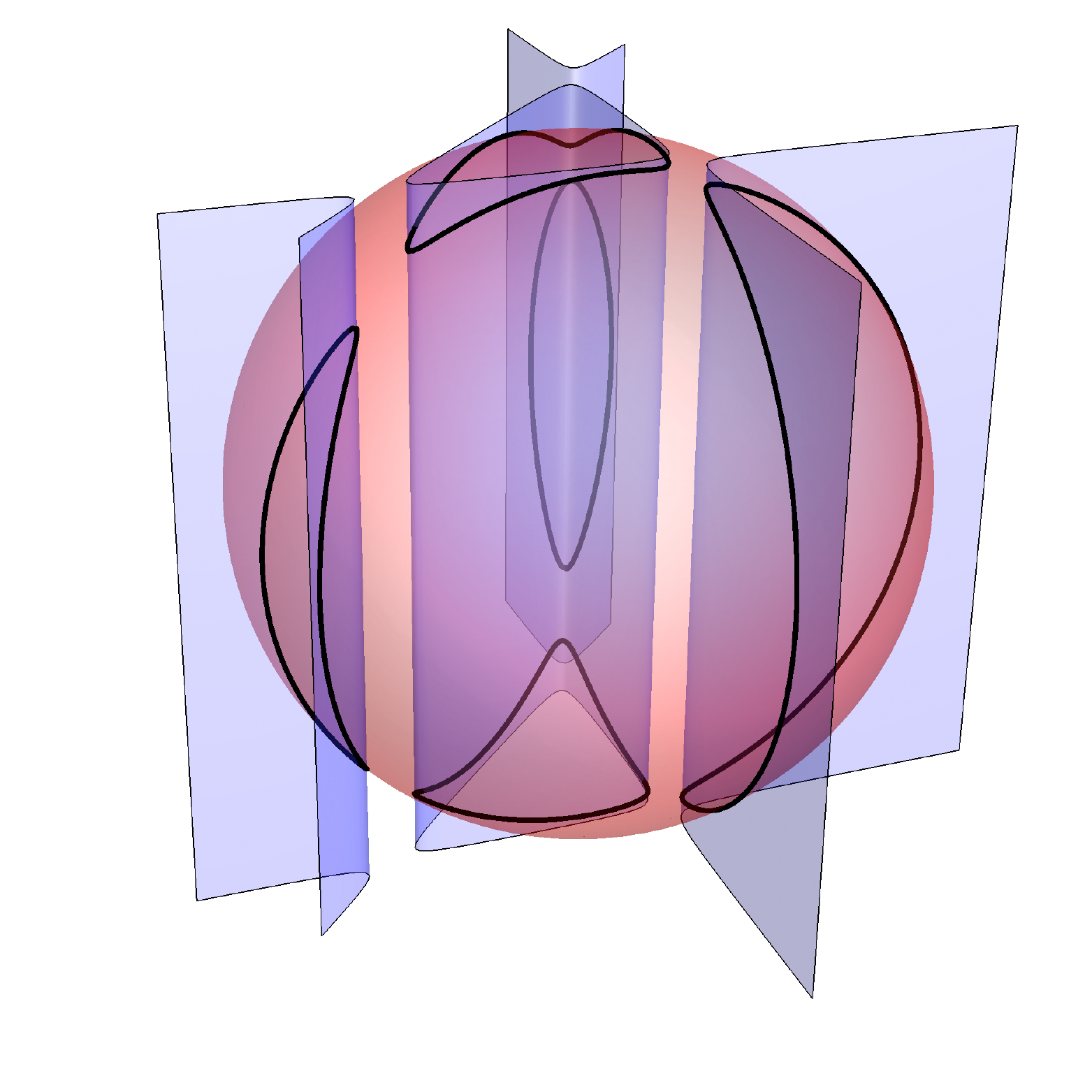}
    \qquad
    \includegraphics[width=4.3cm]{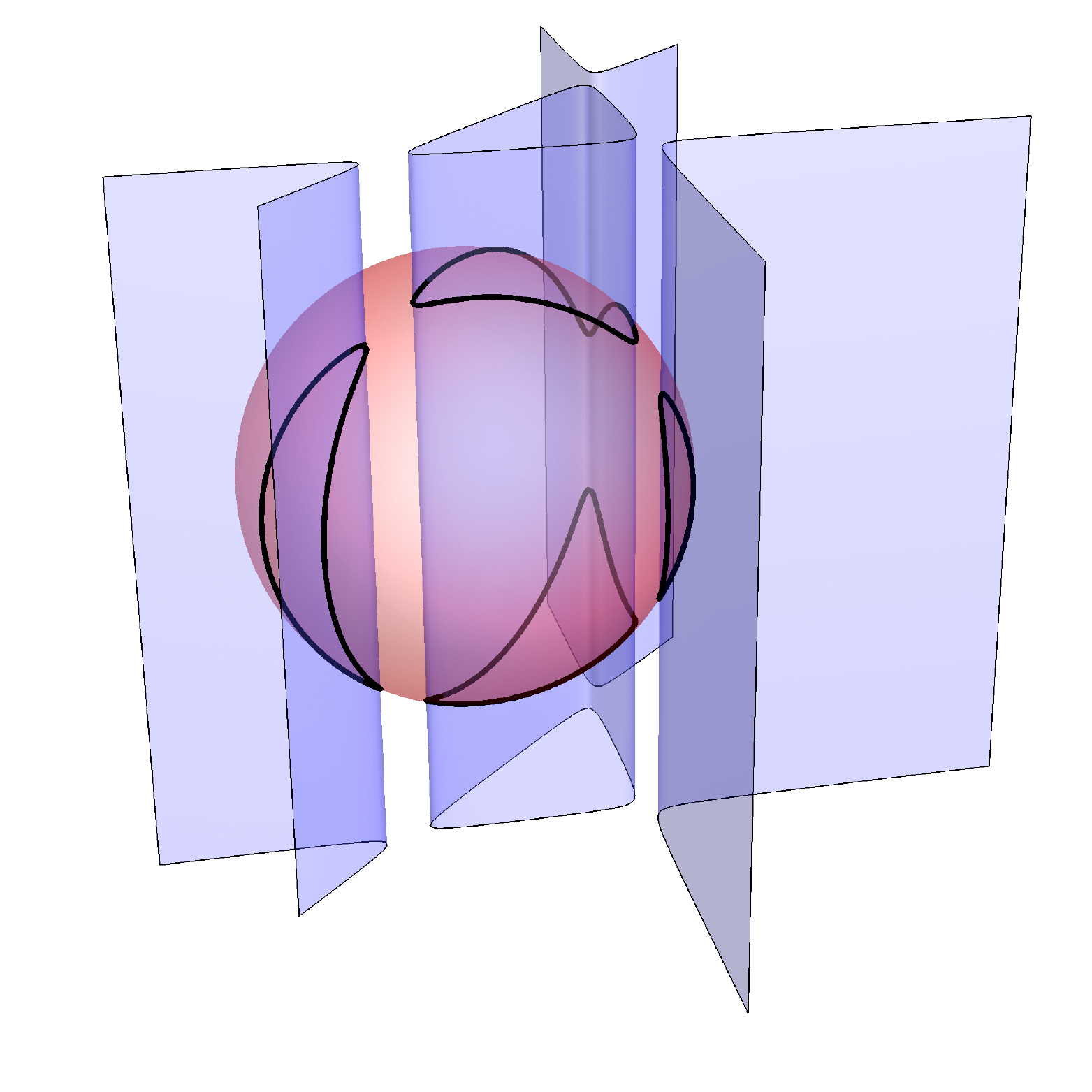}
    \qquad
    \includegraphics[width=4.3cm]{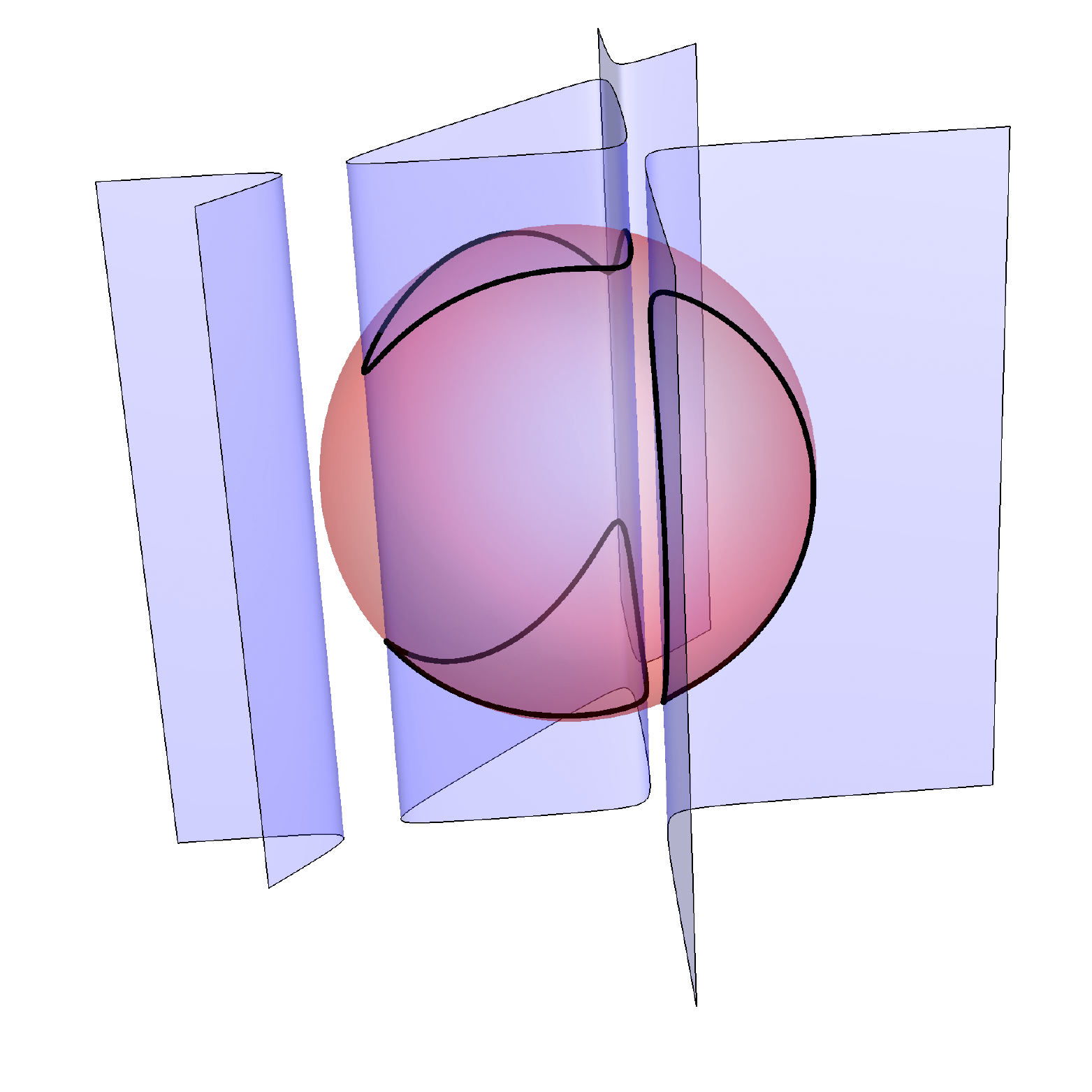}
    \caption{The curves $X_5$, $X_4$, and $X_3$.}
\end{figure}
\noindent Running the algorithm on $I_k$ for a couple of minutes, we
find affine hyperplanes which intersect the curve $X_k$ in real points
only, such as the following three hyperplanes:
\begin{small}
\begin{align*}
  H_5 &= x+ 15307 y-8072  z+6472, \\
  H_4 &= x-14842 y-25786  z-61192, \\
  H_3 &= x+55704  y-26379 z-19751.
\end{align*}
\end{small}\ignorespacesafterend
Each hyperplane $H_k$ intersects $X_k$ in $6$ (distinct) real points.
\begin{figure}[ht]
    \centering
    \includegraphics[width=4.3cm]{pointsX5H5.pdf}
    \qquad
    \includegraphics[width=4.3cm]{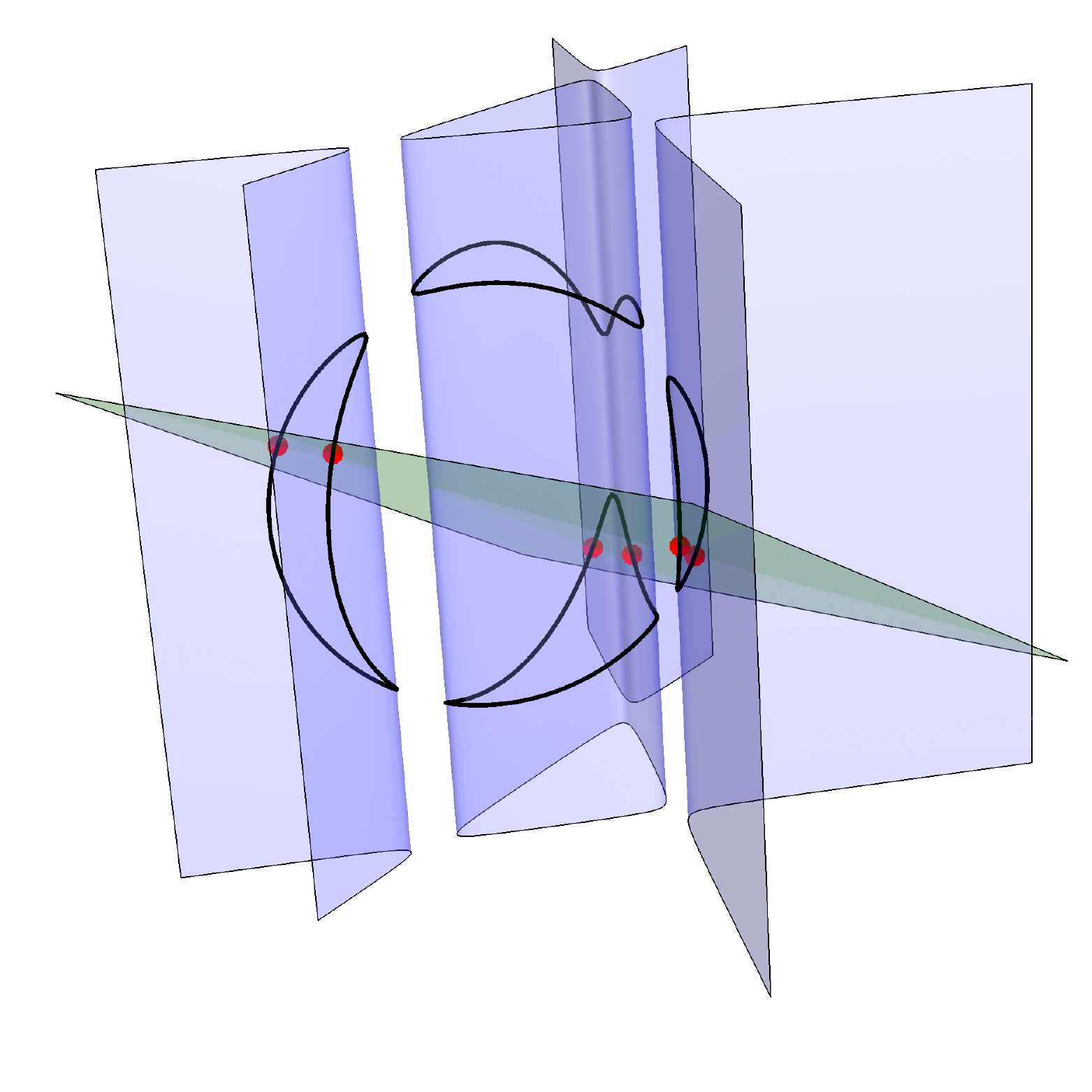}
    \qquad
    \includegraphics[width=4.3cm]{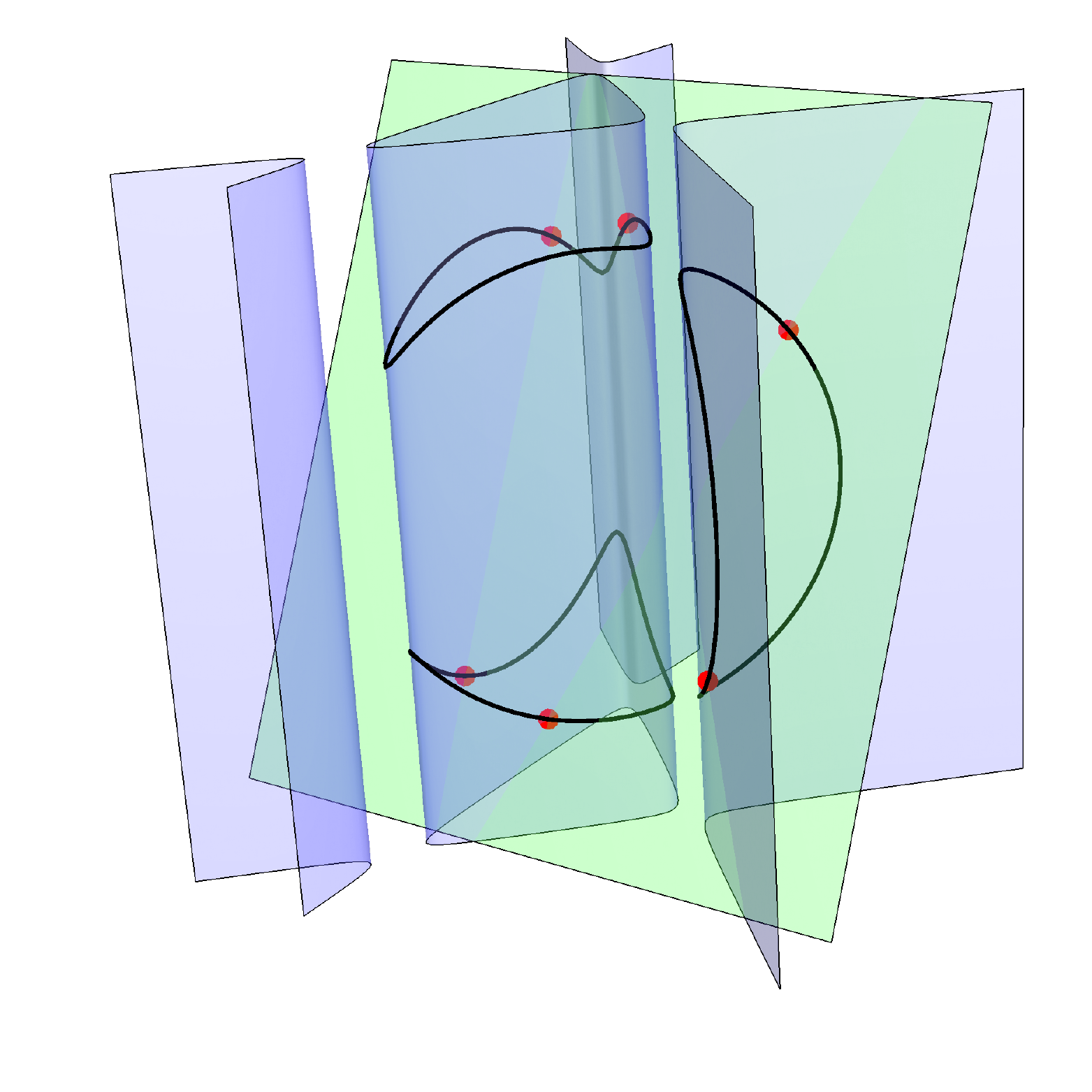}
    \caption{Intersecting curves and planes: $X_i\cap H_i$ for $i=5,4,3$.}
\end{figure}

2. Setting $g_2= x^2+y^2+z^2-10$, let $X_2$ be the projective curve
defined by the affine ideal $I_2 = \langle f,g_2\rangle$. This curve
has $2$ ovals. From the theoretical point of view and in contrast to
the first examples, it is \textit{a priori} not clear whether this curve
possesses a totally real hyperplane section.
Running the algorithm for about $40$ minutes on $I_2$, the
result is that this curve does possess a totally real hyperplane
section. More precisely, the hyperplane
\[H_2 = x+43y/2000+131z/25+9,\]
intersects $X_2$ in $6$ (distinct) real points.
\begin{figure}[ht]
  \centering
  \includegraphics[width=4.3cm]{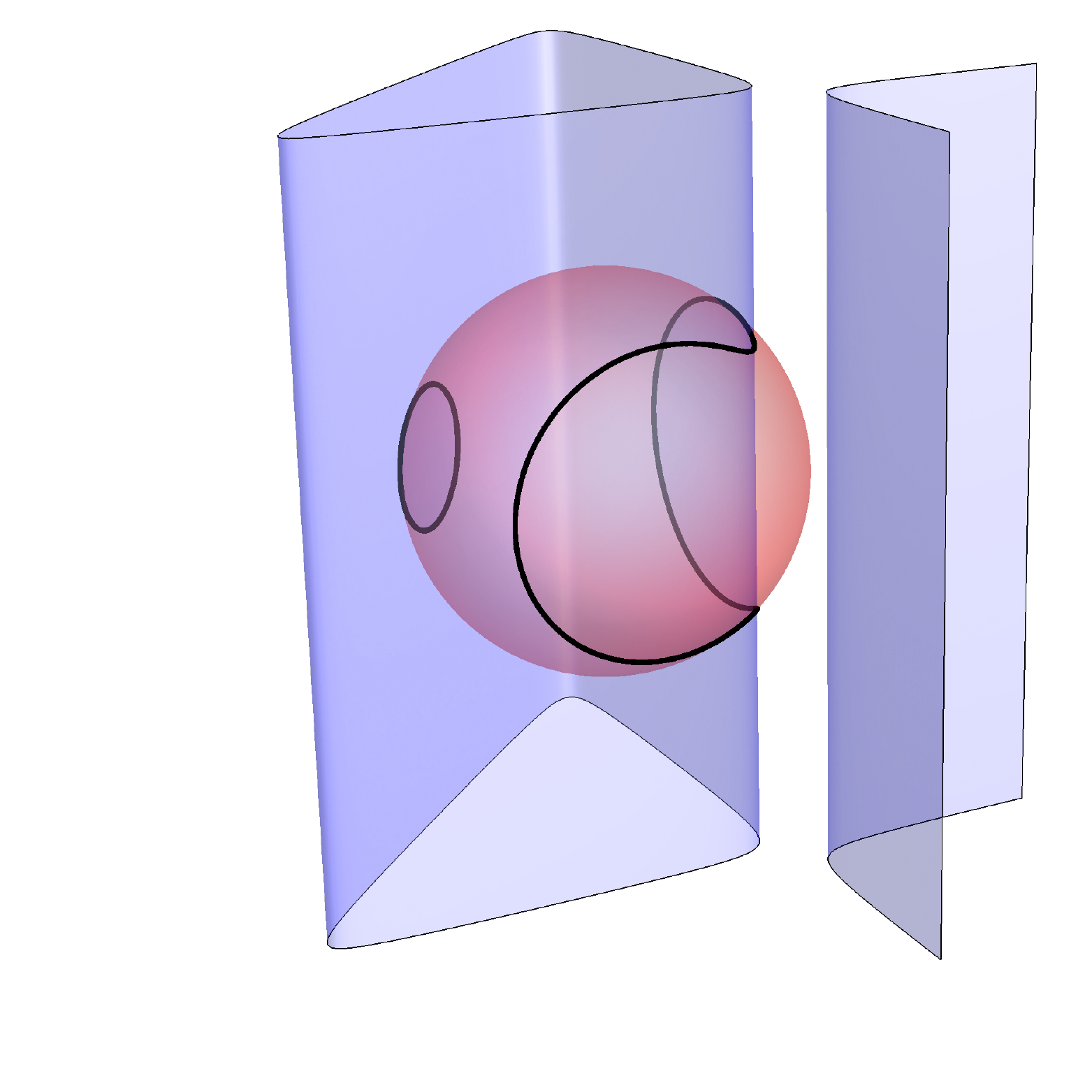}
  \qquad
  \includegraphics[width=4.3cm]{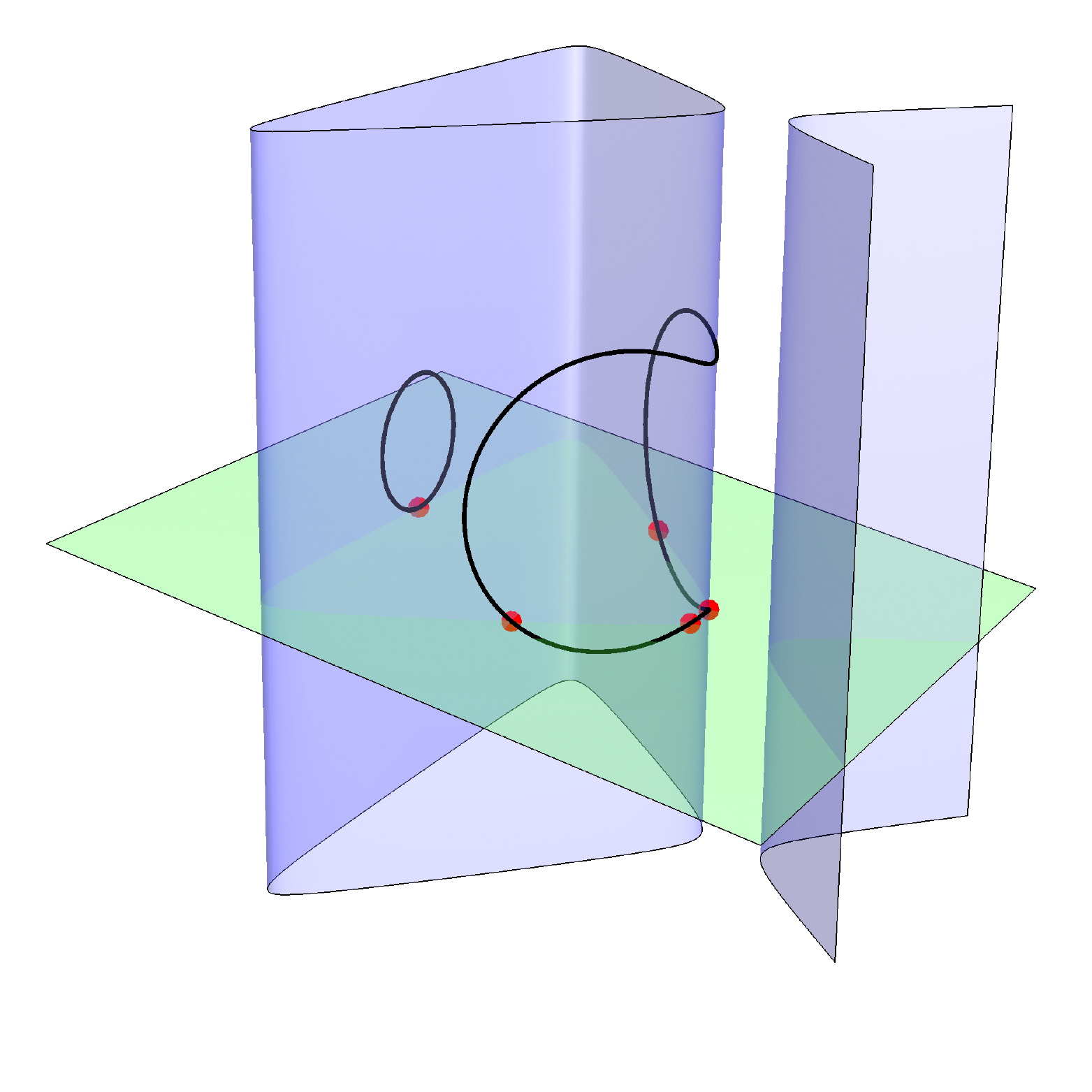}
  \caption{The curve $X_2$ and its intersection with the plane $H_2$.}
\end{figure}

3. Setting $g'_2= (x + 1)^2 + (y + 1)^2 + z^2 - 10$, let $X'_2$ be the
projective curve defined by the affine ideal $I'_2 = \langle
f,g'_2\rangle$. This curve has $2$ ovals, too.
\begin{figure}[ht]
    \centering
    \includegraphics[width=4cm]{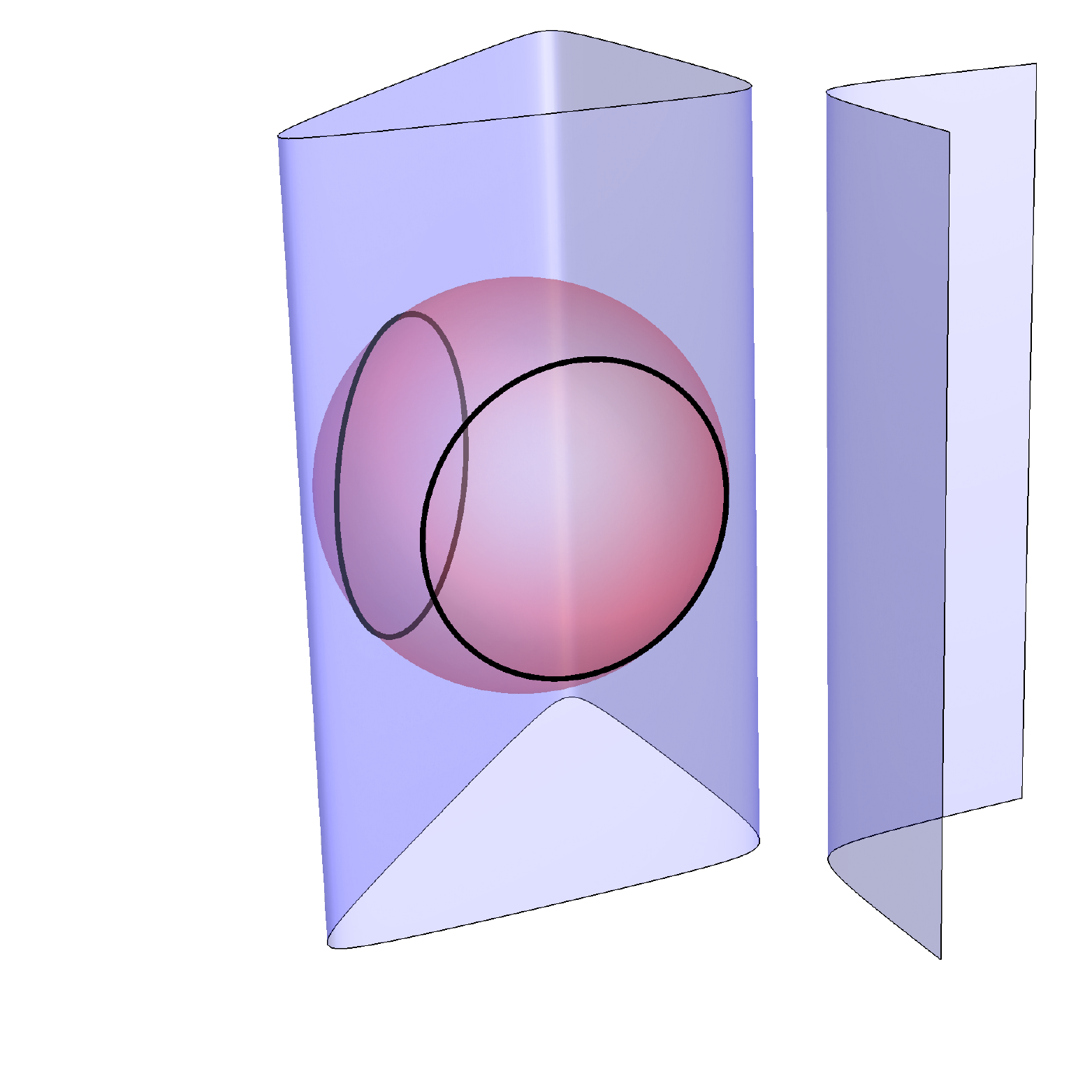}
    \caption{The curve $X'_2$.}
\end{figure}
\noindent We compute a Hermite matrix of size $6\times 6$ in
three parameters, which gives a boundary polynomial $\bm{w}$ of degree
$18$. These computations are done within seconds. The algorithm then
computes points per connected component of the semi-algebraic set
defined by $\bm{w}_{\infty} \cdot \bm{w}_{\cH} \ne 0$. This computation
takes almost $2$ hours. In contrast to the second example, this Hermite
matrix does not attain signature $6$ at any of those points. Besides,
the hyperplanes that correspond to the real solutions of
$\bm{w}_{\infty}$ intersect $X_2'$ at non-real points at
infinity. Thus, these hyperplanes do not give any totally real
hyperplane section. So, $X'_2$ has no simple totally real hyperplane
section. Consequently, we have $N'(X'_2) \geq 7$.

4. For the next example, let us take the Clebsch cubic surface
$f_0=x^3 + y^3 + z^3 + 1 - (x + y + z + 1)^3$ and $g_1=(x + 1)^2 + y^2
+ z^2 - 2$. The projective curve $X_1$ defined by the affine ideal
$I_1=\langle f_0, g_1\rangle$ has only $1$ oval.
The output of the algorithm is the hyperplane
$$H_1=x-4468y-32932z-10164$$ which intersects $X_1$ in $6$ (distinct)
real points.
\begin{figure}[ht]
  \centering
  \includegraphics[width=4cm]{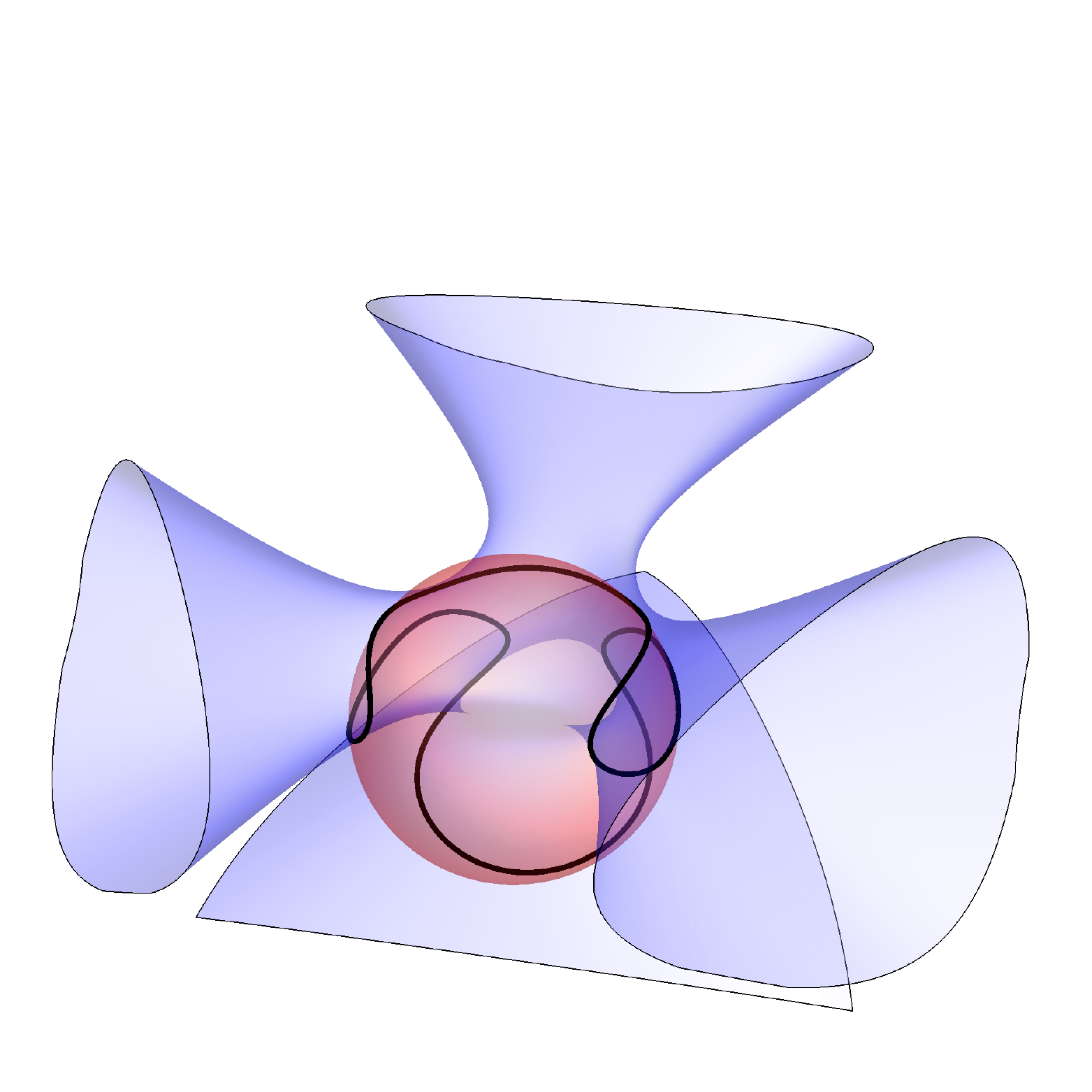}
  \qquad
  \includegraphics[width=4cm]{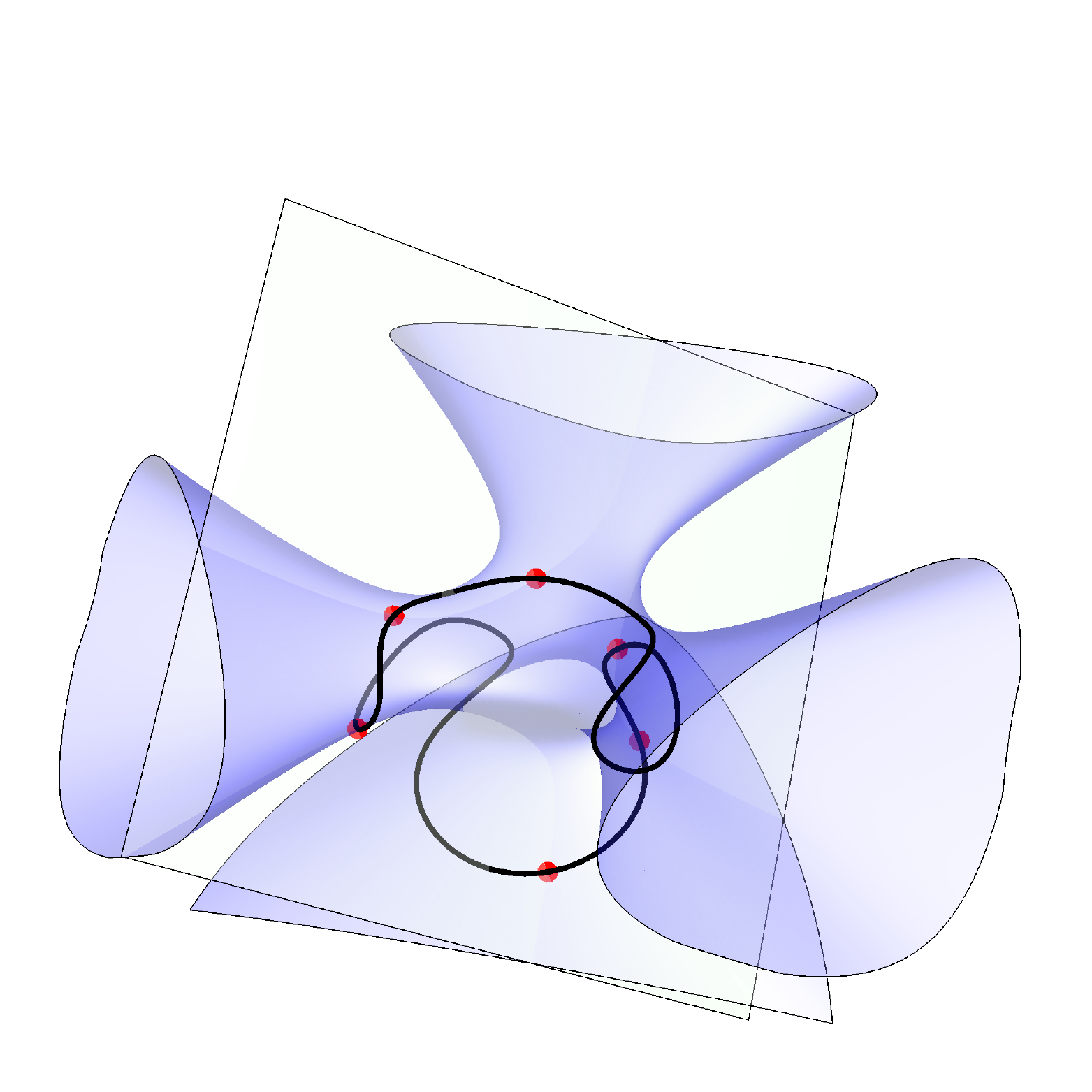}
  \caption{The curve $X_1$ and its intersection with the plane $H_1$.}
\end{figure}

5. Finally, taking $g_1'=(x+2)^2+y^2+z^2-2$, let $X_1$ be the projective curve defined by the affine ideal $I_1'= \langle f,g_1'\rangle$. This curve has only $1$ oval, too.
Again, it is \textit{a priori} not clear whether this curve
has a totally real hyperplane section. 
\begin{figure}[ht]
  \centering
  \includegraphics[width=4.5cm]{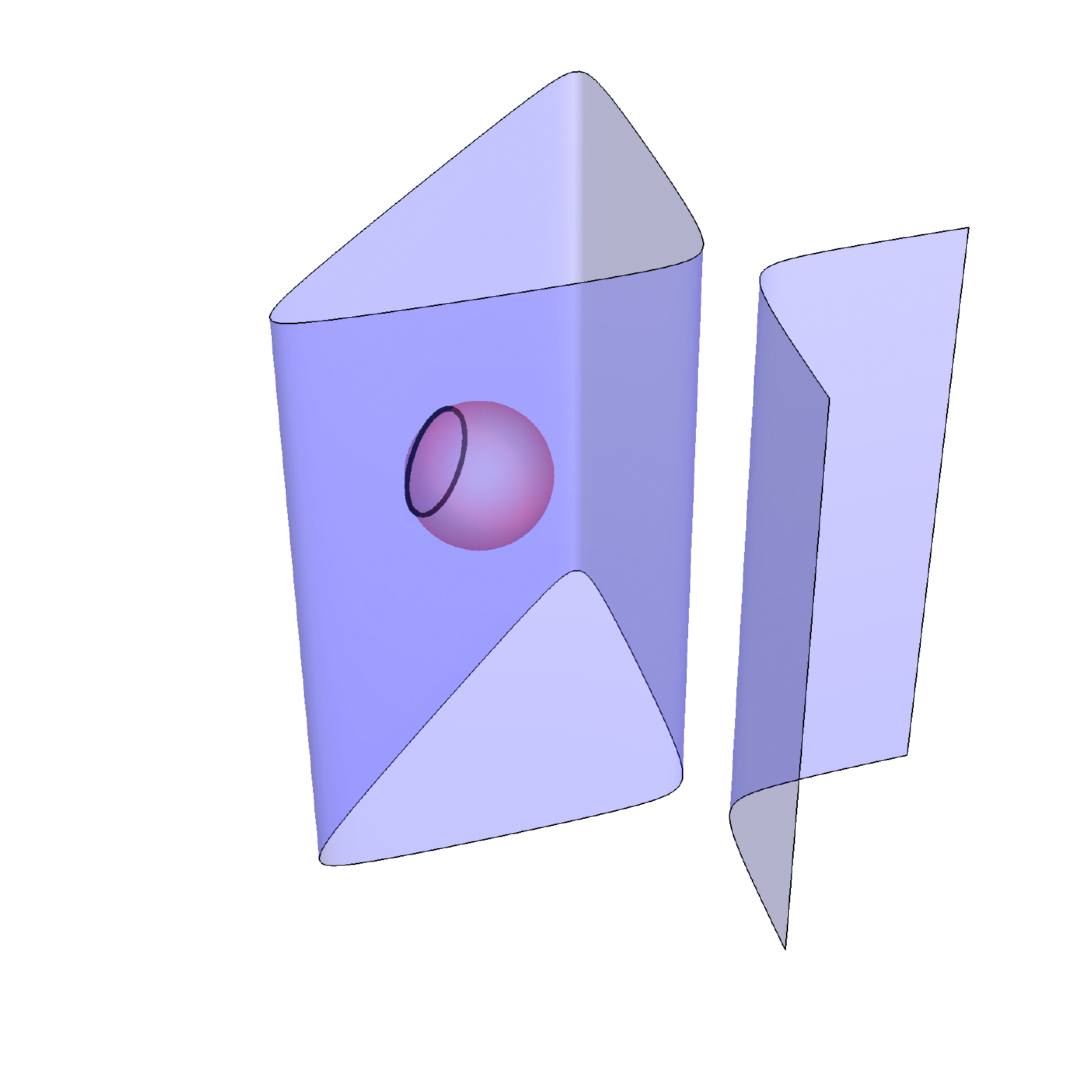}
  \caption{The curve $X_1'$.}
\end{figure}

On this example, our algorithm behaves similarly as in the third
example. We compute a $6\times 6$ Hermite matrix in three parameters
which gives a boundary polynomial $\bm{w}$ of degree $18$. The
computation of sample points of the semi-algebraic set defined by
$\bm{w} \ne 0$ takes $2$ hours and none of the computed sample points
gives the Hermite matrix a signature of $6$. Moreover, the solutions of
$\bm{w}_{\infty}$ here are the same as in the third example and
do not correspond to a totally real hyperplane section. Thus, there is no simple totally real hyperplane section in this
case. Consequently, we have $N'(X_1') \geq 7$. $\hfill\triangle$
\end{ex}

Of course, it takes much effort to show or disprove the existence of a
canonical curve $X$ in $\mathbb{P}^3$ with $1$ or $2$ ovals and $N(X)
\leq 6$. The existence would imply that the real divisor bound $N(X)$
cannot depend on the main topological parameters of a real curve (the
genus, the number of connected components, and whether or not the
curve is of dividing type) only.

As already mentioned, it is a challenging problem to find upper bounds
for $N(X)$ in the case of curves with few branches. However, assuming
the following conjecture by Huisman to be true, Monnier
\cite[Thm.~3.7]{monnier} established new bounds for $(M-2)$-curves
depending on the genus only.

\begin{con}[Conjecture 3.4 in \cite{huisman2}]\label{conjecture}
Let $n \geq 3$ be an odd integer and $X \subset \mathbb{P}^{n}$ be an unramified real curve.
Then $X$ is an $M$-curve and each branch of $X$ is a pseudo-line, i.e., it realizes the non-trivial homology class in $H_{1}(\mathbb{P}^{n}(\mathbb{R}), \mathbb{Z}/2)$.
\end{con}

Recently, a family of counterexamples to Huisman's conjecture has been
constructed for $n=3$ (see \cite{conjectures}). These counterexamples
explicitly contradict the bound found by Monnier in the case of
$g=2$. For our next examples, we briefly recall their construction. A
non-degenerate (i.e., not lying on any real hyperplane) curve $X
\subset \mathbb{P}^n$ is called $\emph{unramified}$ if, taken any real
hyperplane $H$, we have
\[ \text{wt}(H \cdot X) \leq n-1,\]
whereby the weight of the intersection divisor $H\cdot X$ is defined
to be
\[\text{wt}(H \cdot X) =\text{deg}\left(H\cdot X - \left(H \cdot X\right)_{\text{red}}\right),\]
i.e., the degree of the difference between the latter and the reduced
divisor (which contains each point of $H\cap X$ with multiplicity
exactly one). Given two univariate strictly interlacing polynomials
both of degree $d \in \mathbb{N}^{\ast}$, we embed the graph of their
fraction into $\mathbb{P}^3$ via the Segre map. We obtain an
unramified rational curve $C_1$ of degree $d+1$. To obtain a curve of
positive genus, we take a complex-conjugate pair of lines $C_2$ and
consider the union $Z=C_1 \cup C_2$. Taking $\epsilon >0$ small
enough, it is possible to make a small perturbation $Z_\epsilon$ such
that $Z_\epsilon$ becomes a real curve (in particular, we mean smooth
and irreducible) which is unramified. The degree of $Z_\epsilon$ is
$d+3$ and the genus is $2(d-1)$. Since these counterexamples depend on
a parameter $\epsilon >0$, one may wonder whether it is possible to
determine such an $\epsilon >0$ in practice. In the following example,
we reconstruct two such curves and determine different parameters
$\epsilon >0$, for which there exists (and for which there does not
exist) a simple totally real hyperplane section.

\begin{ex}\label{exvar}
For the first example, we consider the same polynomials as in
\cite[Ex.~3]{conjectures}. We obtain a curve of genus $4$, and degree
$6$, which has $1$ oval. In the second example, we construct a
hyperelliptic curve of genus $2$ and degree $5$, which has $1$
pseudo-line.\\

1. Let $q=x_0x_3+x_1x_2$ be the Segre quadric and consider the polynomials
$$ \mbox{\footnotesize $h= 3x_0^3+3x_0x_1^2-x_0^2x_2-3x_0x_2^2+x_2^3+4x_0^2x_3-x_0x_1x_3+4x_1^2x_3-x_2^2x_3-3x_0x_3^2+x_2x_3^2-x_3^3$}$$
and $p= x_0^3+x_1^3+x_2^3-x_3^3$. It is shown in \cite{conjectures} that the curve $X_\epsilon = \mathcal{V}_{+}\left(q, h+ \epsilon p\right)$ does not have a totally real hyperplane section for some small parameter $\epsilon>0$.
On the one hand, the algorithm shows that for $\epsilon =
2^{-4}$, there is a totally real hyperplane section. For example, we
can take the hyperplane
\begin{small}
\begin{align*}
H = -902330031190717857x_0 +& 1152921504606846976x_1\\
+ 323139221492926521 x_2 -& 590264337985175552x_3.
\end{align*}
\end{small}\ignorespacesafterend
On the other hand, for $\epsilon = 2^{-5}$, our algorithm
computes a $6\times 6$ Hermite matrix in three parameters. The
polynomial $\bm{w}_{\infty}$ has two factors: one is linear in the
parameters and the other is a univariate polynomial of degree $3$ in
one parameter.
The boundary polynomial $\bm{w}$ has degree $22$. Computing points per
connected component of the semi-algebraic set defined by $\bm{w} \ne
0$ takes about $4$ hours and does not return any point that gives the
Hermite matrix a signature $6$. 
\begin{figure}[ht]
  \centering
  \includegraphics[width=4cm]{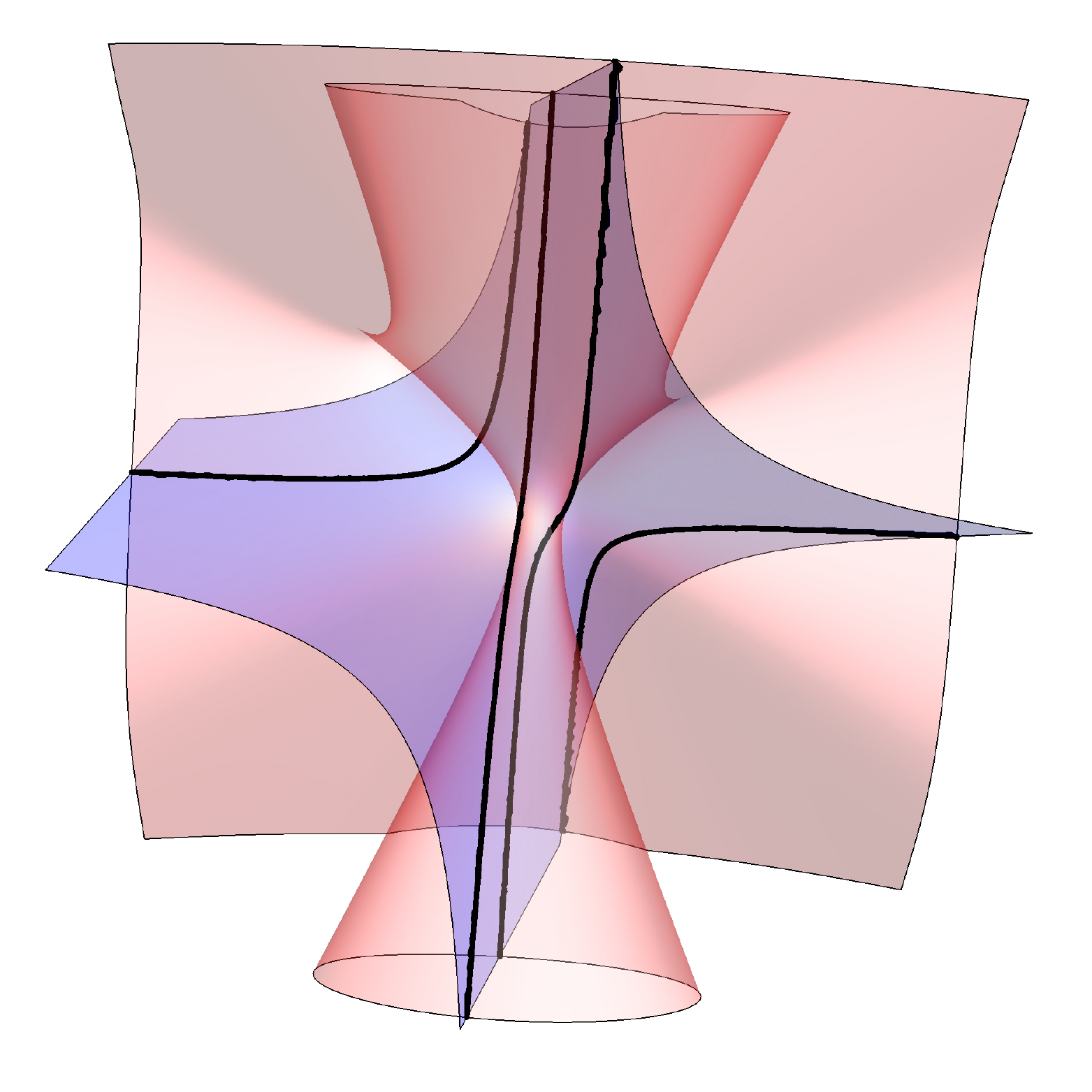}
  \qquad
  \includegraphics[width=4cm]{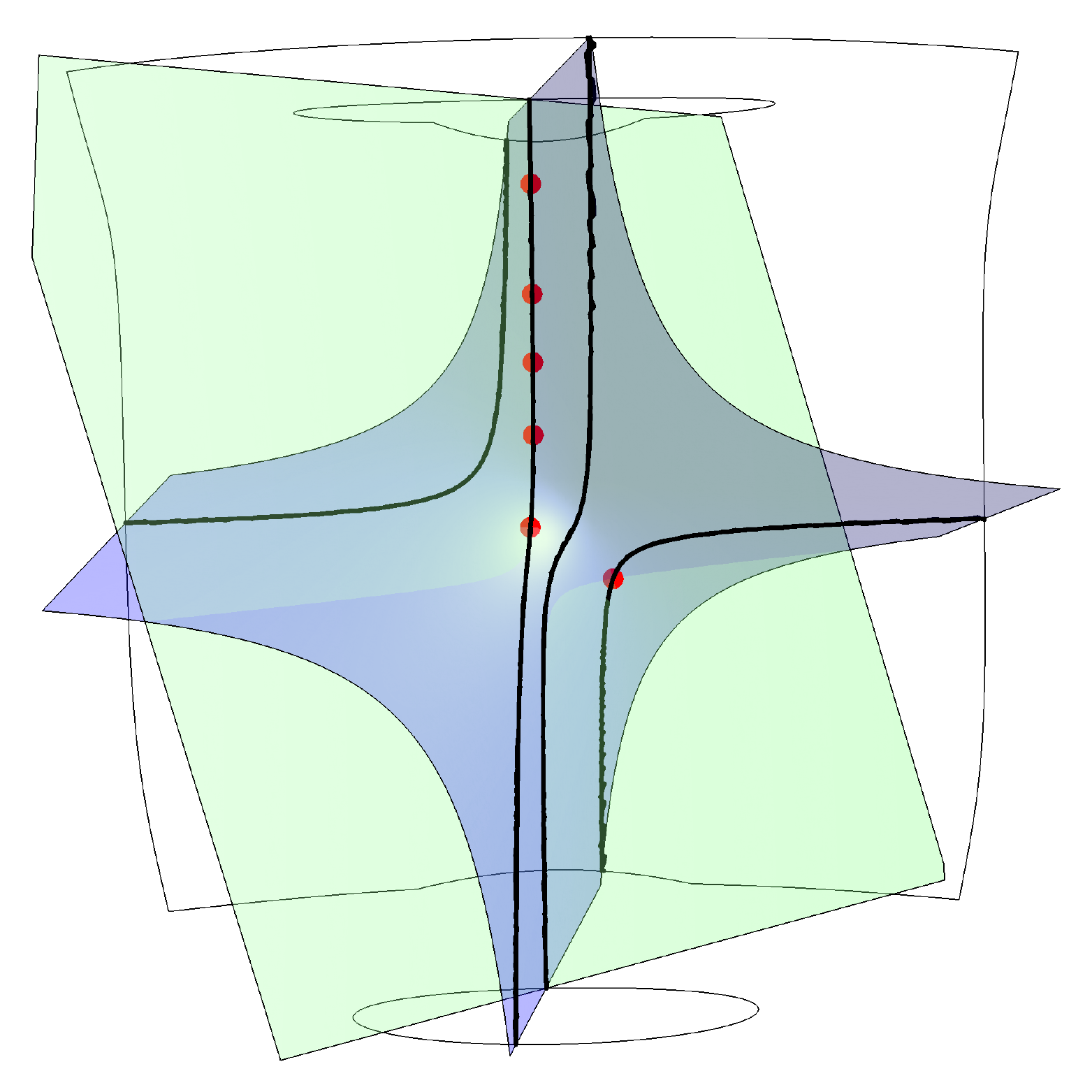}
  \qquad
  \includegraphics[width=4cm]{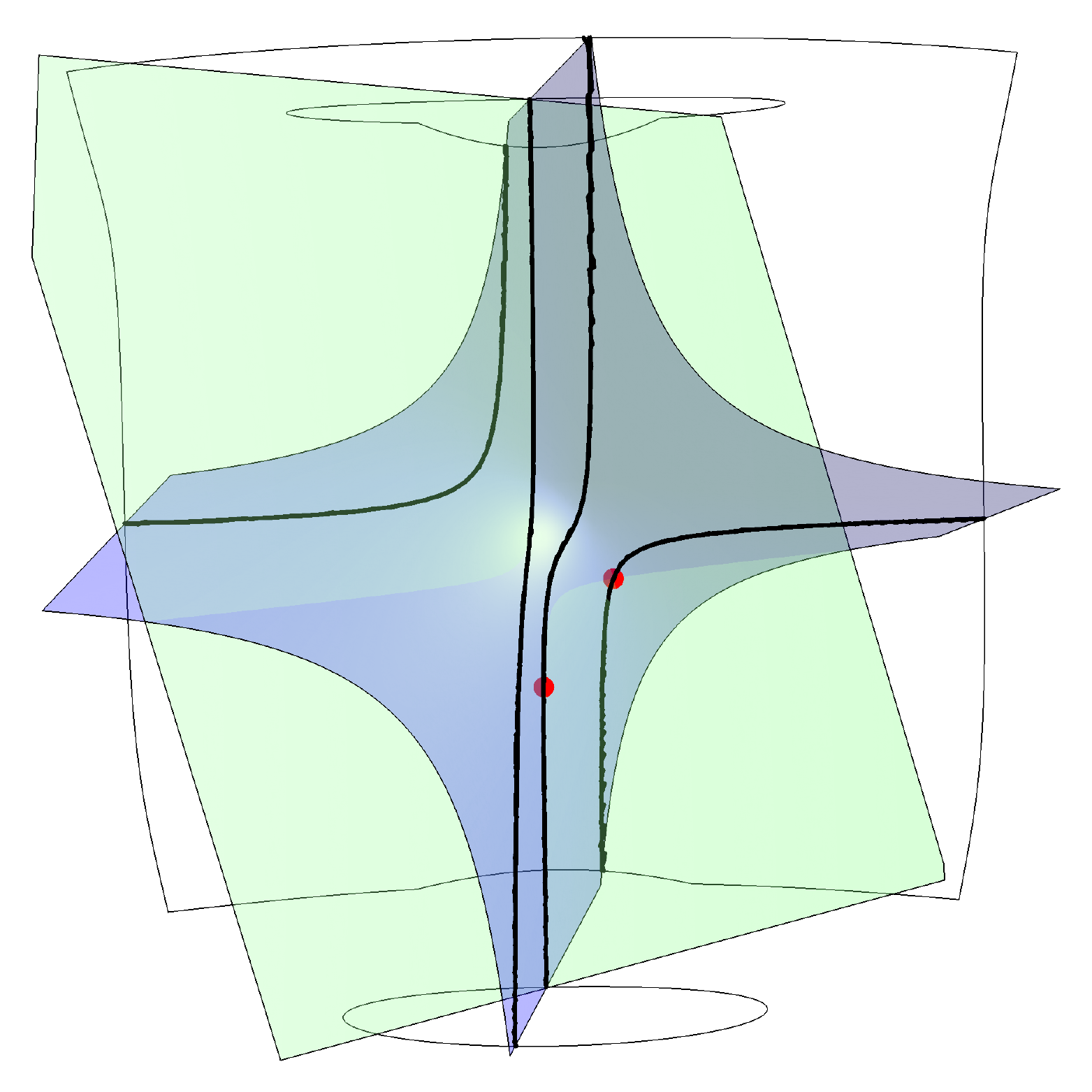}
  \caption{The curve $X_{2^{-4}}$; the intersections $X_{2^{-4}}\cap H$ and $X_{2^{-5}}\cap H$.}
\end{figure}

It remains to classify the solutions when the parameters are real
solutions of $\bm{w}_{\infty}$. For the linear factor, we simply
substitute one parameter by the others in the system to solve and use
the same algorithm (with one less parameter). Finally, we call our
algorithm over the algebraic extension by the univariate factor of
$\bm{w}_{\infty}$ to classify the solutions in this case. These
computations do not return any totally real hyperplane section. So, we conclude that $X_{2^{-5}}$ does not have any simple totally
real hyperplane section. Thus, we have $N'(X_{2^{-5}}) \geq 7$.\\

2. In general, if $X$ is a hyperelliptic curve, then it is known that
$N(X) \geq 2g-1$. If $X$ has at least $g$ branches, then equality
holds (see \cite[Cor.~6.4]{monnier}). Starting with homogeneous
strictly interlacing polynomials $P=y^2-2yx$ and $Q=y^2-x^2$ and
following \cite[Cons.~1]{conjectures}, we can construct a curve of
genus $2$, degree $5$ with $1$ pseudo-line and prescribed intersection
behaviour with any real hyperplane. To be precise, the polynomials
\begin{small}
\begin{align*}
q &= x_0x_3-x_1x_2,\\
f &= -x_0^2x_1-x_1^3+2x_0^2x_2-x_0x_2^2+2x_0x_1x_3 +x_0x_2x_3 -x_0x_3^2+x_1x_3^2,\\
g &= 2x_0x_2^2-x_2^3-x_0^2x_3-x_1^2x_3+x_2^2x_3+2x_0x_3^2-x_2x_3^2+x_3^3,\\
h_1 &= x_0^3+x_1^3+x_0x_2^2-x_1x_3^2,\\
h_2 &= x_0^2x_2+x_1^2x_3+x_2^3-x_3^3
\end{align*}
\end{small}\ignorespacesafterend
define parametrized curves $X_\epsilon = \mathcal{V}_+(q, f+ \epsilon
h_1, g+ \epsilon h_2)$ for $\epsilon >0$. For a small parameter
$\epsilon >0$, the curve $X_\epsilon$ does not have a totally real
hyperplane section. On the one hand, the algorithm shows that for
$\epsilon \in \lbrace 2^{-1}, 2^{-4} \rbrace$, there is a totally real
hyperplane section.

On the other hand, for $\epsilon =2^{-8}$, our algorithm
computes a $5\times 5$ Hermite matrix in three parameters with a
boundary polynomial $\bm{w}$ of degree $15$. Particularly, the
non-specialization polynomial $\bm{w}_{\infty}$ is a product of three
linear polynomials of the parameters. Computing the sample points for
the set defined by $\bm{w} \ne 0$ takes
$3$ minutes and returns no point which gives a signature $5$ to the
Hermite matrix.

When the parameters are real solutions of $\bm{w}_{\infty}$, which has
only linear factors, we substitute one parameter by
the others in the parametric system. This gives us new parametric
systems depending on only two parameters. Using the same algorithm, we
classify the solutions of these new systems and obtain no totally real
hyperplane section when $\bm{w}_{\infty} = 0$. So, we conclude that there is no simple totally real hyperplane
section for $X_{2^{-8}}$. Thus, we have $N'(X_{2^{-8}}) \geq 6$.
$\hfill\triangle$
\end{ex}

From the above examples, we also raise the question of determining the
largest value $\epsilon_0 \in \mathbb{R}$ such that, for any $\epsilon
\in ]0,\epsilon_0[$, the curve $X_{\epsilon}$ has no totally real
hyperplane section. This computation can also be carried out by the
algorithm we present in Section~\ref{Sec:Algorithms} but $\epsilon$ is
now considered as a parameter. However, the boundary polynomial
depends on $4$ indeterminates and has degree up to $35$. So, the
computation of sample points becomes much more difficult.

It remains an open problem to find (or disprove the existence) of a
curve $X$ of genus $2$ with $1$ branch satisfying $N(X)\leq
5$. Furthermore, it remains an unsolved task to find curves with the same
topological parameters, but different values for $N(X)$ or $N'(X)$.
\section{Plane quartics}\label{Sec:planequartics}
Let $X \subset \mathbb{P}^2$ be a plane quartic curve. If $X$ has many branches, i.e., if $s \in \lbrace 3,4 \rbrace$, we know that $4 \leq N(X) \leq  5$. We would expect $N(X)=5$, so we would like to have a possibility to check if certain linear series of degree $4$ do not contain a totally real divisor. The general expectation is $N(X)=2g-1$ for curves of genus $g$ having many branches (see \cite[p.~92]{huisman1}). If $D$ is a divisor of degree $4$ on $X$ having odd degree on at least one branch of $X$, then $\vert D \vert$ can be shown to be totally real. Hence, we are interested in divisors of degree $4$ having even degree on every branch. For such a divisor $D$, there are two possibilities. If $D$ is special, then $\vert D \vert$ is the canonical linear series and must be totally real. If $D$ is non-special, then $\vert D \vert$ defines a morphism to $\mathbb{P}^1$ and in particular, $D$ cannot be very ample. With the help of the algorithm, we are able to check whether each fibre of $X \rightarrow \mathbb{P}^1$ contains a complex-conjugate pair.

If the plane quartic curve $X$ has $s \in \lbrace 1,2 \rbrace$ ovals, we would like to consider very ample divisors of high degree, which give an embedding into a high-dimensional projective space. In this case, we need to check whether the hyperplane linear series of the embedded curve is totally real. For the computations, one can use the divisor package \cite{divisorpackage} in Macaulay2 \cite{M2}.

\begin{rem}
Given a plane quartic curve $X$ with only one oval, no upper bound for
$N(X)$ is known. For two ovals, it is possible to conclude $N(X) \leq
9$ under the assumption of an unsolved case of \Cref{conjecture}. In
particular, it is interesting to check whether every divisor of degree
$10$ defines a totally real linear series. If not, a new case of the
conjecture is disproved. Since divisors of degree $9$ on plane quartic
curves are very ample, one can use the aforementioned divisor package
in Macaulay2 to compute the embedding into a high-dimensional
projective space. Then, one can check the (non-)existence of a totally
real hyperplane section of the image curve.
\end{rem}

If we take a plane quartic curve $X$ (with $s \in \lbrace 3,4
\rbrace$ branches) and a special divisor $D$ of degree $4$,
then the linear series $\vert D \vert$ defines a morphism $\varphi : X
\rightarrow \mathbb{P}^1$. Using the algorithm, we can check whether
there exists a real point $[c:d] \in \mathbb{P}^1(\mathbb{R})$ which
has a totally real fibre. If so, the linear series $\vert D \vert$ is
totally real. If there is no such a point, then $\vert D \vert$ is not
totally real.

By dehomogenizing the projective point $[c:d]$, our algorithm
is reduced to solving a polynomial system depending on one
parameter. Thus, for these examples, we can obtain a complete root
classification of the system by the additional steps using root
isolating algorithms as mentioned at the end of
Section~\ref{Sec:Algorithms}

\begin{ex}\label{quarticex}
We continue with plane quartic curves with many branches and consider divisors of degree $4$.\\

1.~We can use the method described above to get a lower bound for
$N(X)$ on the curve $X=\mathcal{V}_+(x^4+y^4-z^4)$. The linear
series of lines is an example for a linear series which contains a
totally real divisor, but does not contain a simple totally real
one. Hence, we have $N'(X) \geq 5$. We consider the divisor 
\[D=[1:0:1]+[0:1:i]^{\sigma} + [0:1:1]\]
which defines a morphism 
\[X \rightarrow \mathbb{P}^1,\quad [x:y:z] \mapsto
[xy+xz-yz-z^2:x^2-xz].\]
The algorithm shows that there is no totally real fibre. Even more,
each fibre has of at most $2$ real points. Hence, we have $N(X) \geq
5$.\\

2.~In this example, we construct an explicit plane quartic curve with three ovals and a base-point-free linear series of degree four which is not totally real. Generally, if $X$ is a plane quartic curve and $D$ is a special divisor of degree $4$, then the morphism to $\mathbb{P}^1$ is given by conics. Since the intersection of a quartic and a conic consists of eight points (counted with multiplicity), linear equivalence within $\vert D\vert$ is given by a fraction of two conics having four points in common. Conversely, fixing four (real) points on $X$, we may consider the set of conics going trough these points. The four residual points define a linear series of degree $4$. Our goal is to find a linear series which is not totally real. First, we construct a plane quartic curve $X$ with the desired topology. (There are several ways to achieve this; we use a linear determinantal representation and exploit the relation between the Cayley octad, the number of real bitangents, and the number of branches of $X$; see \cite{cayley}). For example, we can take the equation of $X$ to be
\begin{small}
\begin{align*}
f &= 9x^4 - 30x^3y + 161x^2y^2 - 116xy^3 - 8y^4 + 46x^3z - 80x^2yz + 202xy^2z\\
&-  116y^3z + 59x^2z^2 - 80xyz^2 + 185y^2z^2 -6xz^3 - 50yz^3 - 11z^4.
\end{align*}
\end{small}\ignorespacesafterend
Next, we take the circle $c=x^2 + \left(y - \frac{z}{10}\right)^2 -  \frac{2z^2}{10}$ and fix the four real intersection points.
\begin{figure}[ht]
	\centering
  \includegraphics[width=5cm]{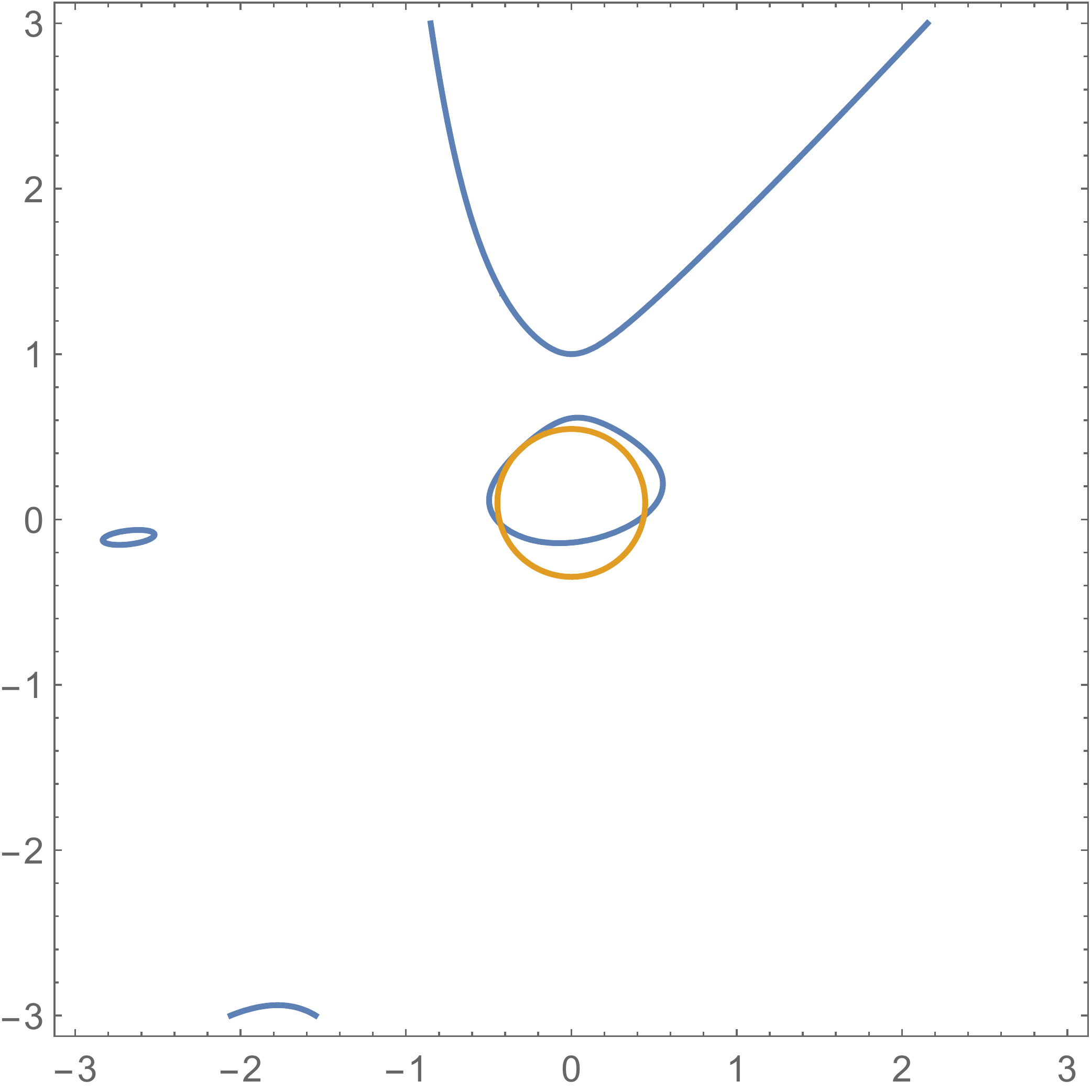}
	\caption{The plane quartic $X$ and the circle $c$.}
\end{figure}
\noindent The real vector space $V=\text{Lin}(Q_1,Q_2)$ of conics
through these points is generated by
\begin{small}
\begin{align*}
Q_1 &= 0.31100521007570264x^2 - 0.4569339120067826xy \\ &+0.7395296982938114y^2 + 0.01692042897825057xz\\
&- 0.3797243325905672yz - 0.05573253113981307z^2,\\
Q_2 &= 0.7303803360779876x^2 + 0.5870985535950933xy\\ &+0.17978406689755905y^2 - 0.021740473005624657xz\\
&+0.2618986086207364yz - 0.14308743118437495z^2.
\end{align*}
\end{small}\ignorespacesafterend
The computational problem is to check whether there is a conic
in $V$ intersecting $X$ in only real points. As in the first example,
we solve a polynomial system of one parameter using the
algorithm of Section~\ref{Sec:Algorithms}.

We start by computing a Hermite matrix of size $8\times 8$ and a
boundary polynomial $\bm{w}$ of degree $24$ ($\deg \bm{w}_{\infty} =
4$, $\deg \bm{w}_{\cH} = 20$). Each fiber over the semi-algebraic set
defined by $\bm{w}\ne 0$ contains $8$ distinct complex points but at
most $6$ real points. 

Next, we isolate the real solutions of $\bm{w}_{\cH}$ and evaluate the
signs of the leading principal minors of $\cH$ at those
solutions. These sign patterns allow us to count the number of real
and complex points at the real solutions of $\bm{w}_{\cH}$. This
handles the case when the parameter takes values that satisfy
$\bm{w}_{\cH} = 0$. For the vanishing locus of $\bm{w}_{\infty}$, we
call the algorithm over its associated algebraic
extension. In both of these cases, we do not find any totally real fiber.

So, our algorithm shows that there is no conic in $V$ intersecting $X$
in real points only. Hence, taking the four residual points of any
intersection $Q \cdot X$ with $Q \in V$ (i.e., leaving the four fixed
points out), we get a divisor of degree four which does not define a
totally real linear series. Furthermore, this linear series is
base-point-free. The plane quartic $X$ is an explicit example where the
bound $N(X)=5$ is determined.\\

3.~Analogously, we can consider the plane quartic curve $X$ defined by
\begin{small}
\begin{align*}
f &= (81x^4)/4 - (135x^3y)/4 + (1953x^2y^2)/16 + (297xy^3)/2 + 69y^4\\
&+ (9x^3z)/2 + (57x^2yz)/2 + (431xy^2z)/8 - (85y^3z)/6 - (179x^2z^2)/4\\
&+ (67xyz^2)/2 - (4685y^2z^2)/48 - (16xz^3)/3 - (1433yz^3)/36 + (917z^4)/36.
\end{align*}
\end{small}\ignorespacesafterend
\noindent The curve $X$ consists of four ovals. 
\begin{figure}[ht]
	\centering
  \includegraphics[width=5cm]{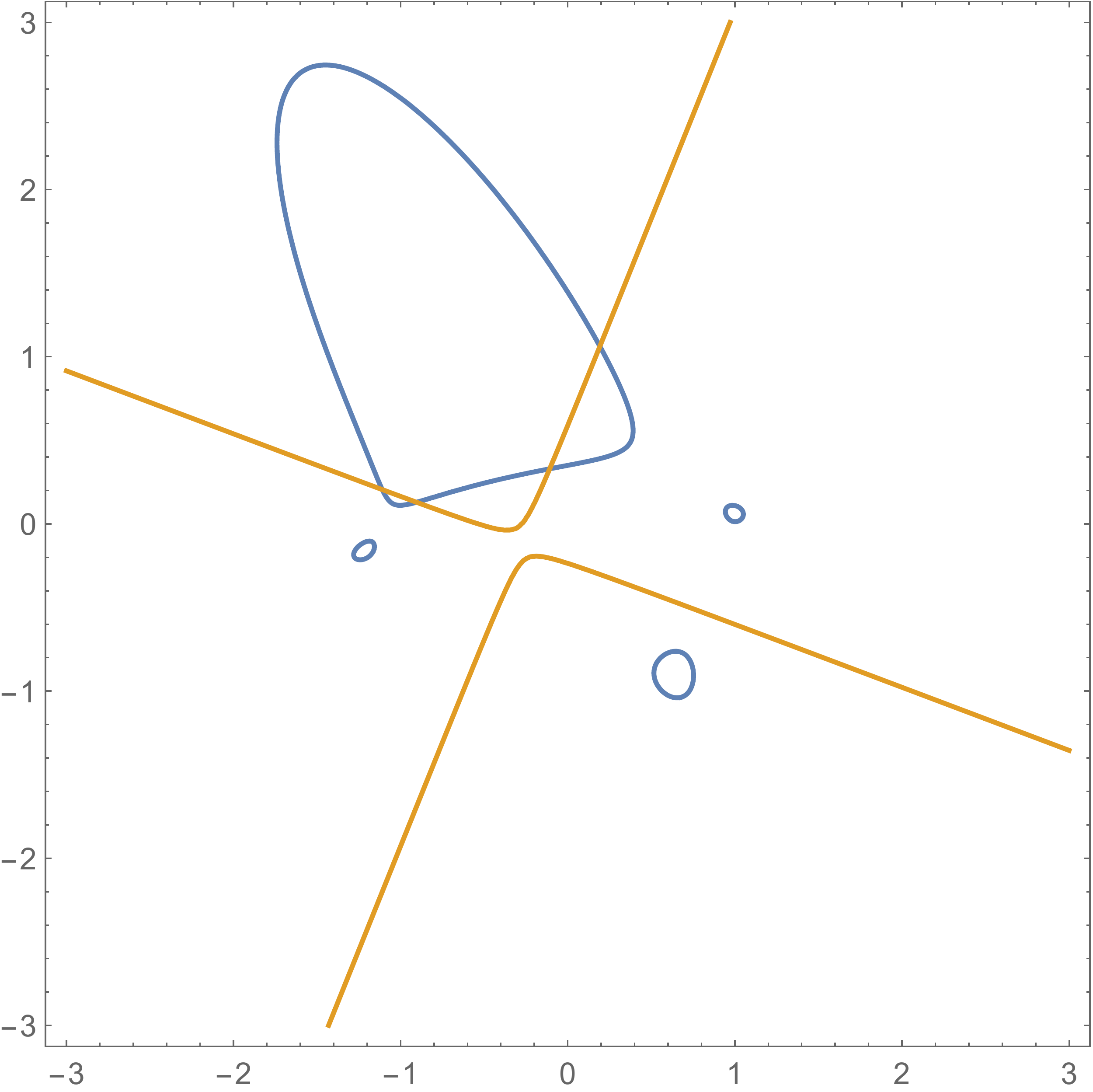}
	\caption{The plane quartic $X$ and conics going through four fixed points.}
\end{figure}
\noindent Summing up, the conics
\begin{small}
\begin{align*}
Q_1 &= 0.47127272928773783x^2 + 0.6598453341260914xy\\
&- 0.13447226903447518y^2 + 0.4868883263821278xz\\ 
&- 0.24467908024400253yz + 0.16581695886185108z^2\\
Q_2 &= -0.09774545786950306x^2 + 0.4442913360602867xy \\
&- 0.5056096052652832y^2 - 0.2532574091360106xz \\
&+ 0.6653828276536204yz - 0.17474649814093252z^2
\end{align*}
\end{small}\ignorespacesafterend
define the real vector space through the four fixed real points.

In this example, our algorithm computes a Hermite matrix $\cH$ of size
$8\times 8$ and a boundary polynomial $\bm{w}$ of degree $20$
($\bm{w}_{\infty} = 1$, $\deg \bm{w}_{\cH} = 20$). Again, the
algorithm shows that there is no conic in this vector space
intersecting $X$ in real points only. Hence, we have $N(X)=5$. $\hfill
\triangle$
\end{ex}

By perturbing the equation of the quartics (and the circles, if necessary), we get infinitely many plane quartics with many components where the real divisor bound is determined.

Increasing the degree, we may ask whether a plane quintic curve $X$
always possesses a totally real line section. If $X$ has $s\geq3$
branches, then there must be exactly one pseudo-line and $s-1$
ovals. Taking a line through two points on two distinct ovals, we
automatically get a totally real line section. Furthermore, we can
conclude that the canonical series is totally real. If $X$ has $s \leq
2$ branches, then the question whether a plane quintic curve possesses
a totally real line section is related to the so-called undulation
invariant (see \cite[Thm.~6.2]{undulation}). Generally, it is possible
to construct plane curves with prescribed topological properties that
have a totally real line section.

\begin{thm}\label{Thm:RealHarnack}
For every $d \geq 3$ and every number $1 \leq s \leq g+1$ with
$g=\frac{(d-1)(d-2)}{2}$, there exists a plane curve $X$ of degree
$d$, genus $g$ and having $s$ branches such that the linear series of
lines $\vert L \vert$ is totally real.
\end{thm}
\begin{proof}
First, we use the method for constructing curves introduced by Harnack \cite[pp.~193-196]{harnack}. For $d=3$, the statement is obvious. Given any $d\geq 4$, he constructs a smooth plane $M$-curve of degree $d$ such that there is a line $L$ intersecting a single component of $X$ in $d$ distinct real points. In the process of constructing the $M$-curve of degree $d$ out of the previous one (of degree $d-1$), we use the classical small perturbation theorem (see \cite[Thm.~3.5]{smallperturbation}), which is originally due to Brusotti \cite{brusotti}. Given the line $L$ and the transversal intersection points with the $M$-curve of degree $d$, we can thus choose the shape of the arcs when smoothing the nodal points. Hence, we can obtain any number of connected components while keeping a line intersecting the resulting curve of degree $d+1$ in $d+1$ distinct real points.
\end{proof}

\begin{cor}
For every $d \geq 4$ and every number $1 \leq s \leq g+1$ with $g=\frac{(d-1)(d-2)}{2}$, there exists a plane curve $X$ of degree $d$, genus $g$ and having $s$ branches such that the canonical series $\vert K \vert$ is totally real.
\end{cor}

One may ask whether it is possible to construct a plane curve $X$ of degree $d \geq 6$ with prescribed topological behaviour such that the linear series of lines is not totally real.

Finally, we remark that our algorithm also works for singular curves. In the case of a singular curve, we only allow the support of the divisors to be contained in the regular locus (see \cite{singularmonnier} for details) , hence it is possible to look for (generic) simple totally real hyperplane sections.

  \newpage

\begin{thebibliography}{10}

  \bibitem{bardet}
  A.~Bardet.
  \newblock {\em Diviseurs sur les courbes r{\'e}elles}.
  \newblock PhD thesis, Universit{\'e} d'Angers, 2013.
  
  \bibitem{BPR}
  S.~Basu, R.~Pollack, and M.-F. Roy.
  \newblock {\em Algorithms in Real Algebraic Geometry (Algorithms and
    Computation in Mathematics)}.
  \newblock Springer-Verlag, Berlin, Heidelberg, 2006.
  
  \bibitem{brusotti}
  L.~Brusotti.
  \newblock {\em Sulla piccola variazione di una curva piana algebrica reale}.
  \newblock 1921.
  
  \bibitem{M2}
  D.~R. Grayson and M.~E. Stillman.
  \newblock Macaulay2, a software system for research in algebraic geometry.
  \newblock Available at \url{http://www.math.uiuc.edu/Macaulay2/}.
  
  \bibitem{principles}
  P.~Griffiths and J.~Harris.
  \newblock {\em Principles of algebraic geometry}.
  \newblock Wiley Classics Library. John Wiley \& Sons, Inc., New York, 1994.
  \newblock Reprint of the 1978 original.
  
  \bibitem{harnack}
  A.~Harnack.
  \newblock Ueber die {V}ieltheiligkeit der ebenen algebraischen {C}urven.
  \newblock {\em Math. Ann.}, 10(2):189--198, 1876.
  
  \bibitem{Hermite56}
  C.~Hermite.
  \newblock Sur le nombre des racines d'une {\'e}quation alg{\'e}brique comprises
    entre des limites donn{\'e}es. extrait d'une lettre {\'a} m. borchardt.
  \newblock {\em J. Reine Angew. Math.}, 52:39--51, 1856.
  
  \bibitem{huisman1}
  J.~Huisman.
  \newblock On the geometry of algebraic curves having many real components.
  \newblock {\em Rev. Mat. Complut.}, 14(1):83--92, 2001.
  
  \bibitem{huisman2}
  J.~Huisman.
  \newblock Non-special divisors on real algebraic curves and embeddings into
    real projective spaces.
  \newblock {\em Ann. Mat. Pura Appl. (4)}, 182(1):21--35, 2003.
  
  \bibitem{smallperturbation}
  I.~Itenberg, G.~Mikhalkin, and J.~Rau.
  \newblock Rational quintics in the real plane.
  \newblock {\em Trans. Amer. Math. Soc.}, 370(1):131--196, 2018.
  
  \bibitem{KoRoSa16}
  A.~Kobel, F.~Rouillier, and M.~Sagraloff.
  \newblock Computing real roots of real polynomials ... and now for real!
  \newblock In {\em Proceedings of the ACM on International Symposium on Symbolic
    and Algebraic Computation}, ISSAC '16, page 303–310, New York, NY, USA,
    2016. Association for Computing Machinery.
  
  \bibitem{krasnov}
  V.~A. Krasnov.
  \newblock Albanese mapping for {$GM{\textbf Z}$}-varieties.
  \newblock {\em Mat. Zametki}, 35(5):739--747, 1984.
  
  \bibitem{conjectures}
  M.~Kummer and D.~Manevich.
  \newblock On {H}uisman's conjectures about unramified real curves.
  \newblock {\em Preprint arXiv:1909.09601}, 2019.
  
  \bibitem{LeSa20}
  H.~P. Le and M.~Safey El~Din.
  \newblock Solving parametric systems of polynomial equations over the reals
    through hermite matrices.
  \newblock {\em Preprint arXiv:2011.14136}, 2020.
  
  \bibitem{liu}
  Q.~Liu.
  \newblock {\em Algebraic geometry and arithmetic curves}, volume~6 of {\em
    Oxford Graduate Texts in Mathematics}.
  \newblock Oxford University Press, Oxford, 2002.
  \newblock Translated from the French by Reinie Ern\'{e}, Oxford Science
    Publications.
  
  \bibitem{monnier}
  J.-P. Monnier.
  \newblock Divisors on real curves.
  \newblock {\em Adv. Geom.}, 3(3):339--360, 2003.
  
  \bibitem{singularmonnier}
  J.-P. Monnier.
  \newblock On real generalized {J}acobian varieties.
  \newblock {\em J. Pure Appl. Algebra}, 203(1-3):252--274, 2005.
  
  \bibitem{cayley}
  D.~Plaumann, B.~Sturmfels, and C.~Vinzant.
  \newblock Quartic curves and their bitangents.
  \newblock {\em J. Symbolic Comput.}, 46(6):712--733, 2011.
  
  \bibitem{undulation}
  A.~Popolitov and S.~Shakirov.
  \newblock On undulation invariants of plane curves.
  \newblock {\em Michigan Math. J.}, 64(1):143--153, 2015.
  
  \bibitem{SaSc03}
  M.~{Safey El Din} and E.~Schost.
  \newblock Polar varieties and computation of one point in each connected
    component of a smooth real algebraic set.
  \newblock In {\em Proc. of the 2003 Int. Symp. on Symb. and Alg. Comp.}, ISSAC
    ’03, page 224–231, NY, USA, 2003. ACM.
  
  \bibitem{scheiderer}
  C.~Scheiderer.
  \newblock Sums of squares of regular functions on real algebraic varieties.
  \newblock {\em Trans. Amer. Math. Soc.}, 352(3):1039--1069, 2000.
  
  \bibitem{divisorpackage}
  K.~Schwede and Z.~Yang.
  \newblock Divisor package for {M}acaulay2.
  \newblock {\em J. Softw. Algebra Geom.}, 8:87--94, 2018.
  
  \bibitem{XiaYang02}
  B.~Xia and L.~Yang.
  \newblock An algorithm for isolating the real solutions of semi-algebraic
    systems.
  \newblock {\em Journal of Symbolic Computation}, 34(5):461--477, 2002.
  
  \end{thebibliography}
\end{document}